\documentclass[authoryear]{elsarticle}

\usepackage[legalpaper, margin=3.5cm]{geometry}
\usepackage{amsmath}
\usepackage{amsthm}
\usepackage{enumitem}
\usepackage{amssymb}
\usepackage{ytableau}
\usepackage{graphicx}
\usepackage{caption}
\usepackage{afterpage}
\usepackage{hyperref}
\usepackage{tikz}
\usepackage{array}
\usepackage{makecell}
\usepackage{diagbox}
\usepackage{multirow}
\usepackage{multicol}
\usepackage{float}
\usepackage{subcaption}
\usepackage{lineno}

\newtheorem{theorem}{Theorem}[subsection]
\newtheorem*{theorem*}{Theorem}
\newtheorem{lemma}[theorem]{Lemma}
\newtheorem{prop}[theorem]{Proposition}
\newtheorem{mr}{Main result}
\newtheorem{prop*}{Proposition}
\newtheorem{cor}[theorem]{Corollary}
\newtheorem*{cor*}{Corollary}
\theoremstyle{definition}
\newtheorem{definition}[theorem]{Definition}
\newtheorem{ex}[theorem]{Example}
\newtheorem{rem}[theorem]{Remark}
\usepackage{colortbl}
\usepackage{listings}
\newcommand{\function}[5]{\begin{array}{crcl}
#1: & #2 & \longrightarrow & #3 \\

    & #4 & \longmapsto & #5 \end{array}}
\newcommand{\Sym}{\mathrm{Sym}}

\DeclareMathOperator{\Id}{Id}
\DeclareMathOperator{\D}{D}

\begin{document}

\begin{frontmatter}

    \title{Reconstruction of hypersurfaces from their invariants}

    \author{Thomas Bouchet}

    \ead{thomas.bouchet@univ-cotedazur.fr}

    \affiliation{organization={Université Côte d'Azur},
                addressline={Campus Sciences, Parc Valrose, 28 Avenue Valrose},
                postcode={06108},
                city={Nice},
                country={France}}

    \begin{abstract}
        Let $K$ be a field of characteristic $0$. We present an explicit algorithm that, given the invariants of a generic homogeneous polynomial $f$ under the linear action of $\mathrm{GL}_n$ or $\mathrm{SL}_n$, returns a polynomial differing from $f$ 
        only by a linear change of variables with coefficients in a finite extension of $K$.
        Our approach uses the theory of covariants and the Veronese embeddings to characterize 
        the linear equivalence class of a homogeneous polynomial through equations whose coefficients are invariants.
        As applications, we derive explicit formulas for reconstructing of a generic non-hyperelliptic curve of genus 4 from its invariants, as well as reconstructing generic non-hyperelliptic curves of genus 3 from their Dixmier-Ohno invariants.
        In both cases, the coefficients of the reconstructed curve lie in its field of moduli. 
    \end{abstract}
    
    \setcounter{tocdepth}{3}
        
\end{frontmatter}
    

\section{Introduction}

Invariant theory is the study of algebraic expressions that remain unchanged under various transformations,
providing powerful tools to analyze and reveal intrinsic features of mathematical structures.

In the 19th century, invariants theorists sought to classify the orbits of homogeneous polynomials in $n$ variables under the action of $\mathrm{SL}_n(\mathbb{C})$.
This naturally led to the study of invariants, which are polynomials in the coefficients of these forms that remain unchanged under the group action.
Using Gordan's algorithm~\citep{gordan}, they constructed generating systems for the algebras of invariants of binary forms ($n = 2$) of degrees $d\leq 6$ and $d = 8$~\citep{sylvester}.
Additionally, they found generating sets of invariants for cubic ternary and quaternary forms~\citep{clebsch}.
Unfortunately, due to the rapidly growing complexity of the task, only partial results were obtained at that time for larger degrees.

The field experienced a resurgence with Mumford's Geometric Invariant Theory (GIT)~\citep{mumford}, which constructs geometric quotients through algebras of invariants.
For instance, knowing a generating set of invariants (and their relations) for the algebra of $\mathrm{Sym}^d(K^n)$ under the action of $\mathrm{SL}_{n}$ enables the explicit construction of 
the geometric quotient of the subvariety of stable elements of $\mathrm{Sym}^d(K^n)$ under $\mathrm{SL}_{n}$, and provides parameters for the corresponding coarse moduli space.
Broadly speaking, a stable element can be represented by a list of invariants. We explain further in the introduction the meaning of the word stable.

In recent years, attention has turned to the reverse problem: 
given the invariants of a stable element of $\mathrm{Sym}^d(K^n)$ under the action of $\mathrm{SL}_{n}$, 
is there an explicit and efficient method to reconstruct a representative of that orbit? 
Although this problem is theoretically solvable using Gröbner bases, such methods are often computationally impractical in practice,
given the large number of invariants and their degrees.

Only specific cases have been addressed in the litterature, and the reconstruction problem does not seem to have been approached in full generality.~\cite{mestre} presented an algorithm for reconstructing binary forms of degree $6$,
while~\cite{noordsij} later introduced an algorithm for the
reconstruction of binary forms of degree $5$ in his Master's thesis. Their methods extend to generic binary forms of even and odd degrees respectively. 
The case of hyperelliptic curves of genus $3$, which are closely related to binary forms of degree $8$,
is treated by~\cite{LR}, even in the presence of extra automorphisms.
\cite{LRS} tackled the reconstruction of plane quartics by reducing the problem to reconstructing a space of binary forms.
Notably, all these methods rely on formulas derived from the work of~\cite{clebsch}.

In this paper, we present a solution to the reconstruction problem which is valid in a very general setting.
This method holds significant potential for applications,
among which the construction of curves or hypersurfaces with interesting arithmetic or geometric properties (e.g. CM curves~\citep{bouyer, kilicer}), arithmetic statistics~\citep{LRR},
and mechanical physics~\citep{olive_rec}. Moreover, this construction could shed light on rationality questions for some moduli spaces of hypersurfaces or curves.
\medskip

Let $n$ and $d$ be positive integers, and let $K$ be an algebraically closed field of characteristic $0$ or $p > d$. Take $W$ as a $(n+1)$-dimensional $K$ vector space
with basis $w_0,\ldots,w_n$ and dual basis $x_0,\ldots, x_n$, and define $\Sym^d(W^*)$
as the space of $(n+1)$-ary $d$-forms with coefficients in
$K$, which is of dimension $\binom{n+d}{d}$.

Let $G = \mathrm{SL}_{n+1}$ or $\mathrm{GL}_{n+1}$.
The group $G$ acts naturally on $W$ by left multiplication \[(g,v)\mapsto gv\,,\]
which induces a contragredient $G$-action on $W^*$
\[(g,x)\mapsto {}^tg^{-1}x,\] where $x$ is written in coordinates in the basis $x_0,\ldots,x_n$.

These actions extend to $\Sym^d(W)$ and $\Sym^d(W^*)$, and we write \[M\cdot f\] for the action of $M\in G$ on $f\in\Sym^d(W)$ or $\Sym^d(W^*)$.
Moreover, we say that $f,f'\in\Sym^d(W^*)$ are $G$-\textit{equivalent} if one can transform $f$ into $f'$ with the action of an element of $G$.
We call \textit{stabilizer} $G_f$ of a form $f\in \Sym^d(W^*)$ the subgroup of matrices $M\in G$ that satisfy 
\[M\cdot f = f.\]

For $G = \mathrm{GL}_{n+1}$, we always have \[U_d = \{\mu I_{n+1}~\lvert~\mu\in K,~\mu^d = 1\} \subseteq G_f.\]

For $G = \mathrm{SL}_{n+1}$, we have a similar inclusion \[U_d = \{\mu I_{n+1}~\lvert~\mu\in K,~\mu^d = 1,~\mu^{n+1} = 1\} \subseteq G_f.\]

We say that $G_f/U_d$ is the \textit{reduced stabilizer} of $f$. The reduced stabilizer does not depend on the choice of $G \in\{\mathrm{SL}_{n+1},\mathrm{GL}_{n+1}\}$.
\medskip

We let $K[\Sym^d(W^*)]^{\mathrm{SL}_{n+1}}$ be the algebra of invariant functions on $\Sym^d(W^*)$
for the action of $\mathrm{SL}_{n+1}$. Since $\mathrm{SL}_{n+1}$ is reductive, a celebrated theorem of~\cite{nagata} states that this algebra of invariants is finitely generated.
Let us assume that a finite set of generators $\{I_j\}_{j\in J}$ of $K[\Sym^d(W^*)]^{\mathrm{SL}_{n+1}}$ is known.
\medskip

We define the \textit{nullcone} \[\mathcal{N}_{\Sym^d(W^*)}^{\mathrm{SL}_{n+1}} = \{f\in \Sym^d(W^*)~|~\forall I\in K[\Sym^d(W^*)]_{>0}^{\mathrm{SL}_{n+1}},~I(f) = 0\}\,.\]
Morally speaking, the nullcone is composed of the degenerate orbits of the variety.

Let us define \[\Sym^d(W^*)^{ss} = \Sym^d(W^*) \backslash \mathcal{N}_{\Sym^d(W^*)}^{\mathrm{SL}_{n+1}}.\] 
The elements of ${\Sym^d(W^*)}^{ss}$ are called \textit{semistable}.
Finally, we let ${\Sym^d(W^*)}^s$ denote the subvariety of ${\Sym^d(W^*)}^{ss}$ consisting of the semistable elements with a closed orbit and finite stabilizer. 
Such elements are called \textit{stable}. For instance, homogeneous forms for $n\geq 2$ and $d \geq 3$ which define smooth hypersurfaces are stable~\cite[Chapter 8]{dolgachev}.
\medskip

Our problem is as follows: let $f\in \Sym^d(W^*)$ be stable, and suppose we are given only its orbit under the action of $\mathrm{GL}_{n+1}$ or $\mathrm{SL}_{n+1}$ as a list of invariants $(I_j(f))_{j\in J}$.
Can we explicitly find an element $f'\in\Sym^d(W^*)$ with the same invariants as $f$?

In this paper, we present an efficient algorithm which solves this problem generically.
In some favorable cases, such as binary forms of odd degrees, or ternary forms of degrees not multiple of $3$, 
the coefficients of the form $f'$ are invariants~(see Section~\ref{sec:examples}). 
\medskip

Unlike previous approaches that rely on Clebsch formulas~\citep{clebsch},
our method is built upon the theory of covariants, linear algebra, and classical tools of algebraic geometry.
We are able to characterize a general form $f$ by equations whose coefficients are invariants, or quotients of invariants,
in a larger space obtained through a Veronese morphism.
These results are established in Theorem~\ref{thm:general} and its Corollary~\ref{cor:general_reconstruction},
with the key ingredient being a Taylor-like identity (Corollary~\ref{cor:decompo_f}).

\begin{mr}[see Corollary~\ref{cor:general}]
    Let $k$ and $n$ be positive integers, and let $K$ be an algebraically closed field of characteristic $0$ or $p>d$. Let $W$ be a $K$-vector space of dimension $n+1$. 
    There exists an efficient algorithm, which, given a list of generating invariants $(I_j(f))_{j\in J}$ corresponding to the orbit of some unknown $f\in\Sym^{k}(W^*)$ which satisfies mild assumptions, returns a form $f'\in\Sym^{k}(W^*)$ with those invariants.
\end{mr}


In general, we require $K$ to be algebraically closed because the reconstructed form $f'$ may a priori lie in a
finite extension of the base field of the invariants. 
However, in practice, we aim to keep these extensions as small as possible (see Remark~\ref{rem:control_extension}).
The case where the covariants have degree $1$ is of great interest and is discussed in Remark~\ref{rem:order1}.
In this case, the reconstruction yields a polynomial whose coefficients lie in the same field as the invariants.
\medskip

It remains to determine when the assumptions of Theorem~\ref{thm:general} are satisfied.

\begin{mr}[see Proposition~\ref{prop:covs}]
    Let $n$ and $k$ be positive integers.
    We let $K$ be an algebraically closed field of characteristic $0$, and $W$ be a $K$-vector space of dimension $n+1$.
    Our algorithm can reconstruct any list of invariants $(I_j(f))_{j\in J}$ corresponding to a stable $f\in\Sym^{k}(W^*)$ with trivial reduced stabilizer.  
\end{mr}

In particular, we get a strong result for the reconstruction of hypersurfaces.

\begin{mr}[see Corollary~\ref{cor:smooth_hyp}]
    Let $n\geq 2$ and $k\geq 3$ be integers. We let $K$ be an algebraically closed field of characteristic $0$, and $W$ a $K$-vector space of dimension $n+1$.
    Let $f\in\Sym^{k}(W^*)$. Our algorithm can reconstruct any list of invariants $(I_j(f))_{j\in J}$ corresponding to a form $f$ which defines a smooth hypersurface without automorphisms. 
\end{mr}

Note that the condition on the stabilizer is not a necessary one, as illustrated by~\cite{cardona} in the case of binary forms of degree 6 with stabilizer $C_2$,
which can be reconstructed using Mestre's algorithm with a different set of covariants than the generic case.
\medskip

Finally, we revisit the reconstruction of binary forms and plane quartics in Section~\ref{sec:examples},
and extend our algorithm to reconstruct non-hyperelliptic curves of genus 4. In these cases, the formulas are remarkably simple, and the reconstructed equations have coefficients in the field where the invariants lie,
or an extension of degree at most $2$ for binary forms of even degrees.

Eventhough this algorithm works in all generality, there are very few instances where it can be effectively applied.
Indeed, generators of the rings of invariants $K[\Sym^d(W^*)]^{\mathrm{SL}_{n+1}}$ are not known in general, except for small values of $d$ and $n$. 
\medskip

We provide a \textsc{magma}~\citep{magma} package for the reconstruction of generic non-hyperelliptic curves
of genus $3$ and $4$~\citep{git-reconstruction}.
To a generic tuple of Dixmier-Ohno invariants, the function \texttt{ReconstructionGenus3} returns a plane quartic with these invariants (up to some weighted projective equivalence).
Similarly, to a generic tuple of invariants of non-hyperelliptic genus $4$ curves~\citep{bouchet}, the function \texttt{ReconstructionGenus4} returns a quadratic form $Q$ and a cubic form $E$ such that
the non-hyperelliptic curve of genus $4$ canonically embedded in $\mathbb{P}^3$ defined by $Q$ and $E$ has said invariants. Both cases $\mathrm{rk}(Q) = 3,4$ are covered,
as well as the reconstruction of hyperelliptic curves of genus 4 from the $106$ invariants exhibited by~\cite{popoviciu}.

\section{Reconstruction}
\label{sec:theory}

In this section, we introduce the building blocks of the paper and expose a key identity (Corollary~\ref{cor:decompo_f}).
We then move on to invariant theory, and apply the previously established results to prove the main Theorem~\ref{thm:general}.
\medskip

Let $n$, $d > 0$ be integers, and let $K$ be an algebraically closed field.
Let $W$ be a $K$-vector space with basis $w_0,\ldots,w_n$, and let $x_0, \ldots, x_n\in W^*$ denote its dual basis.

\subsection{Preliminaries}
\mbox{}

In this section, we introduce a bilinear operator $\D$ which is equivariant in some sense.
This operator can produce new covariants/contravariants from old ones (see Lemma~\ref{lem:D_cov}).

\begin{definition}\label{def:D}
    We extend the natural pairing $W\times W^* \rightarrow K$ to the $K$-bilinear map
    \[\function{\D}{\Sym(W)\times \Sym(W^*)}{\Sym(W^*)}{(w_0^{\alpha_0}\cdots w_n^{\alpha_n}, P)}{\frac{\displaystyle\partial^{\alpha}P}{\displaystyle\partial x_0^{\alpha_0}\cdots \displaystyle\partial x_n^{\alpha_n}}}\,,\]
    where $\alpha = \sum_i\alpha_i$.
\end{definition}

We also define $\D : \Sym(W^*)\times \Sym(W) \rightarrow \Sym(W)$ in a symmetric way.
The order of the arguments resolves any ambiguity. The bilinear map $\D$ is
classically called the \textit{apolarity bilinear form}~\cite{ehrenborg,dolgachev_classical}, and gives an isomorphism $\Sym^d(W^*)\simeq \Sym^d(W)^*$ over a field of characteristic $0$ or $p > d$.

\begin{rem}
    This map can be used to tackle the Waring problem for forms~\cite{ehrenborg}.
    The Waring problem consists, given a form $f\in\Sym^d(W^*)$, in finding the minimal number of linear forms such that $f$
    can be written the sum of the $d$-th powers of these linear forms. For example, a generic ternary quintic can be written as the sum
    of $7$ fifth powers~\cite[Corollary 4.3]{ehrenborg}.        
\end{rem}

\begin{definition}
    Let $d > 0$. We assume that $\mathrm{char}(K)$ is either $0$ or $p > d$. Let $q_0, \ldots, q_r$ be a basis of the $K$-vector space $\Sym^d(W^*)$.
    We say that $q_0^*, \ldots, q_r^*\in \Sym^d(W)$ is a dual basis for $q_0,\ldots ,q_r$, if for any $0\leq i,j\leq r$ we have
    \[\D(q_i^*, q_j)= \delta_{i,j}\,,\]
    where $\delta_{i,j}$ is the Kronecker symbol.
\end{definition}

\begin{lemma}\label{lem:mat_inv}
    We assume that $\mathrm{char}(K)$ is either $0$ or $p > d$. Let $p_0,\ldots, p_r\in \Sym^d(W)$ and $q_0, \ldots, q_r\in\Sym^d(W^*)$.
    Then the matrix \[M_{p,q} := \big(\D(p_i, q_j)\big)_{i,j}\] is invertible if and only if $(p_i)_i$ and $(q_j)_j$ are bases of their respective spaces.
\end{lemma}

\begin{lemma}\label{prop:dual_basis_matrix}
    We assume that $\mathrm{char}(K)$ is either $0$ or $p > d$.
    Let $q_0,\ldots, q_r$ be a basis of $\Sym^d(W^*)$, and let $q_0^*,\ldots,q_r^*$ be its dual basis.
    Let $b_i$ denote the $i$-th element of the canonical monomial basis, in lexicographical order, for $0\leq i\leq r$.
    Let $S$ be the change of basis matrix from $(b_i)_i$ to $(q_i)_i$.
    Then ${}^tS^{-1}$ is the change of basis matrix from $(b_i^*)_i$ to $(q_i^*)_i$.
\end{lemma}

\subsection{Main identity}
\mbox{}

Let $k$, $d$, and $n > 0$ be integers. We assume that $\mathrm{char}(K) > kd$ or $\mathrm{char}(K) = 0$.
Now, let $W$ be a $K$-vector space with basis $w_0,\ldots,w_n$, and dual basis $x_0,\ldots, x_n$.
In this paragraph, we show an identity which allows to recover $f\in\Sym^{kd}(W^*)$ from its $D$-pairings with the $k$-products of a basis of $\Sym^d(W)$, identified as elements of $\Sym^{kd}(W)$.

\begin{definition}
    Let $d\geq 1$, and $\alpha_0,\ldots,\alpha_n\geq 0$ be integers of sum $d$. We define the multinomial coefficient associated to $(\alpha_i)_i$ as
    \[\binom{d}{\alpha_0,\ldots,\alpha_n} = \frac{d!}{\alpha_0!\cdots\alpha_n!}\,.\]
\end{definition}

\begin{lemma}
    For any integers $\alpha_0, \ldots, \alpha_n\geq 0$ of sum $kd$, we define \[J_\alpha = \left\{(\beta_{i,j})_{\substack{1\leq i \leq k\\0\leq j \leq n}} \in \mathbb{Z}_{\geq 0}^{k\times (n+1)}~\Bigg\vert~\sum_{l = 0}^{n}\beta_{i, l} = d\,,\ \sum_{l = 1}^{k}\beta_{l, j} = \alpha_j
    \right\},\]
    where $\alpha = (\alpha_i)_{0\leq i\leq n}$.
    We have the following equality:
    \begin{equation}\label{eq:combinatorics}
        \sum_{(\beta_{i, j})\in J_\alpha}\binom{d}{\beta_{1, 0},\ldots,\beta_{1, n}}\cdots\binom{d}{\beta_{k, 0},\ldots,\beta_{k, n}} = \binom{kd}{\alpha_0, \ldots, \alpha_n}.
    \end{equation}
\end{lemma}

\begin{proof}
    The coefficients of $x_0^{\alpha_0}\cdots x_n^{\alpha_n}$ in $(x_0+\ldots+x_n)^{kd}$
    and in $(x_0+\ldots+x_n)^{d}\cdots (x_0+\ldots+x_n)^{d}$ are equal. By computing these numbers, we obtain Equation~\eqref{eq:combinatorics}.
\end{proof}

We now prove a Taylor-like identity, which is at the heart of the algorithm.
\begin{prop}\label{prop:decompo_f_canonical}
    Let $f\in \Sym^{kd}(W^*)$, let $b_0,\ldots, b_r$ denote the canonical monomial basis of $\Sym^d(W^*)$ in lexicographical order, and let $b_0^*, \ldots, b_r^*$ be its dual basis. Then we have
    \begin{equation}
        \label{eq:decompo_f_canonical}
        \frac{(kd)!}{d!^k}f = \sum_{0\leq i_1,\ldots, i_k\leq r}\D(b_{i_1}^*\cdots b_{i_k}^*, f)b_{i_1}\cdots b_{i_k}\,.
    \end{equation}
\end{prop}

\begin{proof}
    Since $\D$ is bilinear, we prove that statement for monomials.
    
    Let $f = x_0^{\alpha_0}\cdots x_n^{\alpha_n}\in~\Sym^{kd}(W^*)$.
    Further, if we write $b_i = \prod_{j=0}^{n}x_j^{\gamma_{i,j}}$, where $\gamma_{i,j}$ is a nonnegative integer,
    then for any integers $0\leq i_1,\ldots,i_k\leq r$, we have \[b_{i_1}\cdots b_{i_k} = \prod_{j = 0}^{n}x_i^{\sum_{l = 1}^{k} \gamma_{i_l,j}}= \prod_{j = 0}^{n}x_i^{\sum_{l = 1}^{k} \beta_{l,j}},\]
    where we set $\beta_{l,j} = \gamma_{i_l, j}$. 

    We compute the right member of Equation~\eqref{eq:decompo_f_canonical}:
    \begin{align*}
        & \sum_{0\leq i_1,\ldots, i_k\leq r}\D(b_{i_1}^*\cdots b_{i_k}^*, f)b_{i_1}\cdots b_{i_k}\\
        &= \sum_{0\leq i_1,\ldots, i_k\leq r}\D(b_{i_1}^*\cdots b_{i_k}^*, f)b_{i_1}\cdots b_{i_k}\\
        &= \sum_{\substack{(\beta_{l, j})\in\mathbb{Z}_{\geq 0}^{k\times (n+1)}\\\forall l,\, \sum_{j = 0}^{n}\beta_{l, j} = d}}\D\left(\frac{1}{\beta_{1, 0}!\cdots \beta_{k, n}!}w_0^{\sum_{l = 1}^{k}\beta_{l, 0}}\cdots w_n^{\sum_{l = 1}^{k}\beta_{l, n}}, f\right)x_0^{\sum_{l = 1}^{k}\beta_{l, 0}}\cdots x_n^{\sum_{l = 1}^{k}\beta_{l, n}}\\
        &= \sum_{(\beta_{l, j})\in J_\alpha}\frac{\alpha_0! \cdots \alpha_n!}{\beta_{1, 0}!\cdots \beta_{k, n}!}f\\
        &= \frac{\alpha_0! \cdots \alpha_n!}{d!^k}\sum_{(\beta_{l, j})\in J_\alpha}\frac{d!}{\prod_{j = 0}^{n}\beta_{1, j}!}\cdots\frac{d!}{\prod_{j = 0}^{n}\beta_{k, j}!}f\\
        &= \frac{\alpha_0! \cdots \alpha_n!}{d!^k}\sum_{(\beta_{l, j})\in J_\alpha}\binom{d}{\beta_{1, 0},\ldots,\beta_{1, n}}\cdots\binom{d}{\beta_{k, 0},\ldots,\beta_{k, n}}f\\
        &= \frac{\alpha_0! \cdots \alpha_n!}{d!^k}\binom{kd}{\alpha_0, \ldots, \alpha_n}f\\
        &= \frac{(kd)!}{d!^k}f\,.
    \end{align*}

    Since $\mathrm{char}(K) > kd$ or $\mathrm{char}(K) = 0$, all the operations above are well-defined.
\end{proof}

A somewhat similar computation is carried out in~\cite[Proposition 3.2]{ehrenborg}.

\begin{rem}
    To avoid problems in positive characteristic, the author also tried to
    use Hasse-Schmidt derivatives~\citep{hasse} instead of partial derivatives.
    However, the factorial numbers do not cancel out with the Hasse Derivative as we might expect,
    thus the conditions on the characteristic of $K$ must remain.
\end{rem}

\begin{cor}\label{cor:decompo_f}
    Let $f\in \Sym^{kd}(W^*)$, let $q_0\ldots,q_r$ be a basis of $\Sym^d(W^*)$, and let $q_0^*, \ldots, q_r^*$ denote its dual basis.
    Then we have
    \begin{equation}\label{eq:decomposition_f}
        \frac{(kd)!}{d!^k}f = \sum_{0\leq i_1,\ldots, i_k\leq r}\D(q_{i_1}^*\cdots q_{i_k}^*, f)q_{i_1}\cdots q_{i_k}\,.
    \end{equation}
\end{cor}

\begin{proof}
    We use Lemma~\ref{prop:dual_basis_matrix} and Proposition~\ref{prop:decompo_f_canonical}
    to compute the right hand side of Equation~\eqref{eq:decomposition_f}.
    After some simplifications, we obtain the desired result.
\end{proof}

Let us introduce the Veronese embedding that maps
  $[x_0:\cdots:x_n]$ to all monomials of total degree $d$:\[\function{v_{n,d}}{\mathbb{P}^{n}}{\mathbb{P}^{r}}{[x_0:\cdots:x_n]}{[x_0^d:x_0^{d-1}x_1:\cdots:x_n^d]}\,.\]
It is well-known that $v_{n,d}$ realizes an isomorphism of $\mathbb{P}^{n}$
onto its image, which is defined by irreducible quadratic forms~\cite[Exercise
2.5]{harris}. Let $X_0,\ldots, X_r$ denote coordinates for $\mathbb{P}^r$. 
These quadratic forms can be written as $X_iX_j-X_lX_m$ for some well-chosen $i$, $j$, $l$, and $m$, where we can have $i=j$ or $l=m$.
The number of linearly independent such quadratic forms is \[\dim(\Sym^2(\Sym^d(W^*)))-\dim(\Sym^{2d}(W^*))\,.\]

Let $q_0\ldots,q_r$ be a basis of $\Sym^d(W^*)$.
It is clear that the
morphism \[\function{\varphi}{\mathbb{P}^{n}}{\mathbb{P}^{r}}{[x_0:\cdots:x_n]}{[q_0:\cdots:q_r]}\]
is also an isomorphism of $\mathbb{P}^{n}$ onto its image, which is defined by irreducible quadratic forms,
which reflect the relations that exist among the $q_i$'s.


\medskip

This paper relies heavily on the following key result.

\begin{prop}\label{prop:recover}
    Let \[\tilde{f} = \sum_{0\leq i_1,\ldots, i_k\leq r}\D(q_{i_1}^*\cdots q_{i_k}^*, f)X_{i_1}\cdots X_{i_k},\]
    and let $Q_0,\ldots, Q_s$ be a set of quadratic forms which define $\mathrm{Im}(\varphi)$.

    Then, the knowledge of $\tilde{f}$ and the $Q_i$'s is enough to recover $f'\in\Sym^{kd}(W^*)$ which is $\mathrm{GL}_{n+1}$-equivalent to $f$.
\end{prop}

\begin{proof}
    We assume that $\varphi$ and $f$ are not known, otherwise the statement becomes trivial.

    We know that there exists a parametrization of $\mathrm{Im}(\varphi)$, because $\varphi$ is one of them. 
    Let $\varphi' : \mathbb{P}^n\longrightarrow\mathbb{P}^r$ be any such parametrization. Since $\varphi$ and $\varphi'$ have the same image, 
    we can consider the automorphism $\varphi^{-1}\circ \varphi'$ of $\mathbb{P}^n$, and it is known that the group of automorphisms of $\mathbb{P}^n$ is $\mathrm{PGL}_{n+1}$~\citep[Exercise 18.7]{harris}.
    Thus, the parametrizations $\varphi$ and $\varphi'$ differ only by an element of $\mathrm{PGL}_{n+1}$.
    
    If we let $q_0',\ldots, q_r'\in\Sym^d(W^*)$ be coordinates of $\varphi'$, then $\tilde{f}(q_0',\ldots,q_r')$ is $\mathrm{GL}_{n+1}$-equivalent to $f$ by the previous analysis.
\end{proof}

\begin{rem}
    The parametrization step in that proof is not constructive.
    To develop an effective and efficient algorithm, we require a constructive approach. 
    Section~\ref{sec:param} addresses this in detail.
\end{rem}

It remains to see how we can write the coefficients of $\tilde{f}$ and the quadratic relations $Q_i$ as invariants. 
The use of a linear basis of covariants of a given degree will be primordial to achieve that goal.

\subsection{Generalization to tensor spaces}
\label{sec:generalization}
\mbox{}

Equation~\eqref{eq:decomposition_f} can be extended to tensor spaces.
Let $W_1,\ldots, W_s$ be finite-dimensional $K$-vector spaces.
Let $\D_s$ denote the composition of the operators $\D$ for $W_1,\ldots, W_s$.
In that situation, a similar statement as Corollary~\ref{cor:decompo_f} can be made.

\begin{prop}\label{prop:decompo_f_tensor}
    Let $k, d_1,\ldots, d_s>0$.
    Let $f\in \Sym^{kd_1}(W_1^*)\otimes \cdots\otimes \Sym^{kd_s}(W_s^*)$, and let $q_0\ldots,q_r$ be a basis of $\Sym^{d_1}(W_1^*)\otimes \cdots\otimes \Sym^{d_s}(W_s^*)$.
    Let $q_0^*, \ldots, q_r^*$ denote its dual basis with respect to $\D_s$ (such a basis exists, and is unique).
    Then we have
    \begin{equation}\label{eq:decomposition_f_tensor}
        \frac{(kd_1)!}{d_1!^k}\cdots \frac{(kd_s)!}{d_s!^k}f = \sum_{0\leq i_1,\ldots, i_k\leq r}\D_s(q_{i_1}^*\cdots q_{i_k}^*, f)q_{i_1}\cdots q_{i_k}\,.
    \end{equation}
\end{prop}

This can be proven by induction on $s$, since $\D$ acts independently and successively on the different spaces $W_i$.

Let us define a morphism $\varphi$, which associates to a point of $\prod_i\mathbb{P}^{\dim (W_i)-1}$ the $q_j$'s evaluated at this point.
Its image is a Segre-Veronese variety, and $\varphi$ realizes an isomorphism of the projective variety $\prod_i\mathbb{P}^{\dim (W_i)-1}$ onto its image.
The algorithmic solution for the parametrization presented in Section~\ref{sec:param} extends naturally to that case.

\subsection{Invariant theory}
\label{sec:invariant_theory}
\mbox{}

In this section, we introduce some notions of invariant theory.
\medskip

Let $n$ and $d$ be positive integers, and let $K$ be an algebraically closed field. Take $W$ as a $(n+1)$-dimensional $K$ vector space.

\begin{definition}\label{def:cov}
    Let $k > 0$ and $r\geq 0$ be integers.
    A \textit{covariant} (resp. \textit{contravariant}) of $\Sym^{k}(W^*)$ of order $r$ is an $\mathrm{SL}(W)$-equivariant homogeneous polynomial map
    \begin{align*}
        &C:\,\Sym^{k}(W^*)\rightarrow \Sym^r(W^*)\\
    \text{(resp. }&C:\,\Sym^{k}(W^*)\rightarrow \Sym^r(W)\,\text{).}
    \end{align*}
    The \textit{degree} $d$ of $C$ is its degree as a homogeneous polynomial map.
    In the special case $r=0$, $C$ is called an \textit{invariant}. Moreover, the
    \textit{weight} (or index) of a covariant the number $(kd-r)/(n+1)$
    (resp. $(kd+r)/(n+1)$).
\end{definition}

\begin{rem}
    In fact, when a covariant is transformed by a matrix $A\in \mathrm{GL}_{n+1}$, we have
    \[C(A\cdot f) = \det(A)^{-\frac{kd-r}{n+1}} (A\cdot C(f)),\]
    for any $f\in\Sym^k(W^*)$.
    In order for this expression to be well-defined, the weight must be an integer. 
    
    Thus, we note that given a value of $k$ and $n$, not all values of $d$ and $r$ give an integer.
    For instance, the case of binary forms of even degrees, corresponding to $k = 2k'$ and $n = 1$, implies that $r$ must be even. 
    Thus there exist no covariants of odd degree for binary forms of even degree.
\end{rem}

\begin{definition}
    Let $f\in\Sym^{k}(W^*)$. We say that the covariants (or contravariants) $q_1,\ldots,q_r$ of order $d$ of $\Sym^{k}(W^*)$ are \textit{linearly independent at} $f$
    if the forms $q_i(f)$ are linearly independent.
    
    We say that the $q_i$'s are \textit{generically linearly independent} if there exists a dense open subvariety \[U\subseteq\Sym^k(W^*)//\mathrm{SL}_{n+1} = \mathrm{Spec}(K[\Sym^k(W^*)]^{\mathrm{SL}_{n+1}})\]
    such that for every $f\in U$, the $q_i$'s are linearly independent at $f$.
\end{definition}

\begin{rem}
    Since the quotient variety $\Sym^k(W^*)//\mathrm{SL}_{n+1} = \mathrm{Spec}(K[\Sym^k(W^*)]^{\mathrm{SL}_{n+1}})$ is affine by definition, and its coordinate ring is a domain,
    the variety $\Sym^k(W^*)//\mathrm{SL}_{n+1}$ is irreducible. Therefore, if there exists one $f\in\Sym^k(W^*)$ such that the $q_i$'s are linearly independent at $f$, it means that the 
    $q_i$'s are generically linearly independent.
\end{rem}

To construct covariants and contravariants, one usually starts with a form whose coefficients are indeterminates,
and then applies equivariant transformations. The apolarity bilinear operator $\D$ turns out to be equivariant in a sense that we will explore (see Lemma~\ref{lem:D_cov}).
Another valuable tool is the transvectant~\citep{olver},
which is traditionally used for the description of the algebra of covariants and invariants of binary forms.

However, in general, it is not possible to construct all covariants and invariants through repeated iterations of the transvectant.
We refer the interested reader to~\cite{kohel}, in which the authors recall several ways to construct covariants and contravariants.
We now show how the apolar bilinear form $\D$ defined in Definition~\ref{def:D} can be used to construct covariants and contravariants.

\begin{definition}
    Let $p$ be a contravariant of order $r_p$ of  $\Sym^{d}(W^*)$, and let $q$ bea covariant of order $r_q$ of $\Sym^{d}(W^*)$.
    We define $\D(p, q)$ pointwise: \[[\D(p, q)](f) = \D(p(f), q(f))\in\Sym^{r_q-r_p}(W^*)\] for all $f\in \Sym^d(W^*)$, with the convention that $\Sym^{-r}(W^*) = \{0\}$ for any positive integer $r$.
    We define symmetrically $\D(p, q)$ by \[[\D(q, p)](f) = \D(q(f), p(f))\in\Sym^{r_p-r_q}(W^*)\] for all $f\in \Sym^d(W^*)$.
\end{definition}

\begin{lemma}\label{lem:D_cov}
    Let $p$ be a contravariant of $\Sym^{d}(W^*)$ and $q$ a covariant of $\Sym^{d}(W^*)$ of respective orders $r_p, r_q$ and degrees $d_p, d_q$.
    Then $\D(p, q)$ (resp. $\D(q,p)$) is a covariant (resp. contravariant) of $\Sym^{d}(W^*)$ of degree $d_p+d_q$.
\end{lemma}

\begin{lemma}\label{lem:dual_basis}
    Let $r = \dim_K(\Sym^d(W^*))-1$, and let $l$ be a positive integer. Let $q_0,\ldots,q_r$ be covariants of order $d$ of $\Sym^l(W^*)$,
    which are generically linearly independent. Let $S$ be the change of
    basis matrix from the canonical basis $(b_i)_i$ of $\Sym^d(W^*)$ to $(q_i)_i$, and let $\Delta$ be its determinant.
    Then $\Delta$ is a non-zero invariant of $\Sym^l(W^*)$, and $\Delta {}^tS^{-1}$ is a matrix whose columns are contravariants, precisely the dual basis $q_0^*,\ldots, q_r^*$ multiplied by the invariant $\Delta$.
\end{lemma}

\subsection{Main theorem}
\label{sec:main_thm}
\mbox{}

We have all the tools at our disposal to present the main results of this paper, Theorem~\ref{thm:general} 
and its Corollary~\ref{cor:general_reconstruction}.

\begin{theorem}\label{thm:general}
    Let $k$, $d$, and $n$ be positive integers. Let $K$ be an algebraically closed field of characteristic $0$ or $p > kd$.
    Let $W$ be a $K$-vector space with basis $w_0,\ldots, w_n$ and dual basis $x_0,\ldots, x_n$.
    Let $f\in\Sym^{kd}(W^*)$, and let $r = \dim_K(\Sym^d(W^*))-1$.
    We assume that there exist $r+1$ covariants of order $d$ of $\Sym^{kd}(W^*)$ which
    are linearly independent at $f$. Let $q_0,\ldots, q_r$ be such covariants.
    Let us define \[\function{\varphi}{\mathbb{P}^{n}}{\mathbb{P}^{r}}{[x_0:\cdots:x_n]}{[q_0(f):\cdots:q_r(f)]}.\]

    Let $X_0, \ldots, X_r$ be coordinates for $\mathbb{P}^{r}$.
    We define \[\tilde{f} = \sum_{0\leq i_1,\ldots, i_k\leq r}\D(\Delta
      q_{i_1}^*\cdots \Delta q_{i_k}^*, \Id)(f)X_{i_1}\cdots X_{i_k},\]
    where $\Id$ is the identity covariant of $\Sym^{kd}(W^*)$, and we let $Q_0,\ldots, Q_s$ be a set of quadratic forms in the $X_i's$ which define $\mathrm{Im}(\varphi)$.

    The coefficients of $\tilde{f}$ are invariants of $\Sym^{kd}(W^*)$, and
    the coefficients of the $Q_i$'s can be chosen to have coefficients which can be computed in terms of invariants of $\Sym^{kd}(W^*)$.
    
    Moreover, the knowledge of $\tilde{f}$ and the $Q_i$'s is enough to recover $f'\in\Sym^{kd}(W^*)$ which is $\mathrm{GL}_{n+1}$-equivalent to $f$.
\end{theorem}

\begin{proof}
    Let us further assume that $\varphi$ and $f$ are unknown (otherwise the statement is trivial).
    Let $\Delta q_0^*,\ldots,\Delta q_r^*$ be the set of contravariants defined in Lemma~\ref{lem:dual_basis} of $\Sym^{kd}(W^*)$.
    By assumption, they are linearly independent at $f$. 

    It is clear by Lemma~\ref{lem:D_cov} that the coefficients of $\tilde{f}$ are invariants of $\Sym^{kd}(W^*)$.
    Remains to see how we can compute quadratic forms defining the image of $\varphi$ with invariants.

    The image of $\varphi$ is defined by quadratic forms that reflect the relations between the $q_i(f)$'s. 
    We note that the family $(\Delta q_i^*(f)\Delta q_j^*(f))_{0\leq i,j\leq r}$ generates the space $\Sym^{2d}(W)$,
    and $\D$ is non-degenerate. Therefore for any $Q\in \Sym^{2d}(W^*)$ we have 
    \[\left(\forall\, 0\leq i, j\leq r,~\D(\Delta q_i^*(f)\Delta q_j^*(f), Q) = 0\right) \iff Q = 0.\]

    Thus, one way to find a basis of quadratic relations for the $q_i(f)$'s is to compute the right kernel of the $(r+1)^2\times (r+1)^2$
    matrix of invariants \[\big(\D(\Delta q_i^*\Delta q_j^*\,,
      q_lq_m)(f)\big)_{\substack{0\leq i,j\leq r\\0\leq l, m\leq r}}.\]

    Finally, the last part of the result is a direct consequence of Proposition~\ref{prop:recover}, applied to the basis of covariants $q_0(f),\ldots,q_r(f)$ and its dual basis (up to $\Delta$) $\Delta q_0^*(f),\ldots, \Delta q_r^*(f)$.
\end{proof}

\begin{cor}\label{cor:general_reconstruction}
    For a general $f\in\Sym^{kd}(W^*)$, knowing only the values of the invariants \[\D(\Delta
      q_{i_1}^*\cdots \Delta q_{i_k}^*, \Id)(f)\] for $0\leq i_1,\ldots, i_k\leq r$
      and \[\big(\D(\Delta q_i^*\Delta q_j^*\,,
      q_lq_m)(f)\big)_{\substack{0\leq i,j\leq r\\0\leq l, m\leq r}},\]
    is theoretically enough to recover a form $f'\in\Sym^{kd}(W^*)$ which is $\mathrm{GL}_{n+1}$-equivalent to $f$.
\end{cor}

In practice, it is hard to find a parametrization $\varphi'$ from the $Q_i$'s in practice. 
We will see an algorithmic solution to this problem for small values of $r$.

\section{Reconstruction algorithm}
\label{sec:reconstruction}

In this section, we present a reconstruction algorithm, which, given a set of specialized generating invariants of $K[\Sym^{kd}(W^*)]^{\mathrm{SL}_{n+1}}$,
returns an element of $\Sym^{kd}(W^*)$ with said invariants. It relies on Theorem~\ref{thm:general}, and a parametrization algorithm.
\medskip

Let $k$, $d$, and $n$ be positive integers. Let $W$ be a $K$-vector space with basis $w_0,\ldots, w_n$ and dual basis $x_0,\ldots, x_n$.
Let $r = \dim_K(\Sym^d(W^*))-1$.
Let us assume that there exist $r+1$ covariants of order $d$ of $\Sym^{kd}(W^*)$ which
are generically linearly independent, and we take $q_0,\ldots, q_r$ to be such covariants.

\subsection{Finding a parametrization}\label{sec:param}
\mbox{}

We turn to the problem of finding a parametrization of $\varphi$ from a set of quadratic forms defining its image,
which is a special instance of a challenging problem, concerned with the parametrization of a projective variety from its implicit representation.
Notably, the reverse problem of finding an implicit representation from a parametrization is an equally interesting question which is also hard to solve efficiently.

Our approach leverages the particular geometry of our problem to give an algorithmic solution. The central idea is that we know a parametrization of the canonical Veronese embedding.
By performing a suitable change of basis, we can reduce the problem to the case where the quadratic forms correspond to those defining the image of the canonical Veronese embedding.
\medskip

Let $q = (q_i)$ be a basis of $\Sym^d(W^*)$, and let $q^* = (q_i^*)$ denote its dual basis.
We define
\[Q_{q^*,q} = \big(\D(q_i^*q_j^*\,,
    q_lq_m)\big)_{\substack{0\leq i,j\leq r\\0\leq l, m\leq r}}.\]

Let $f\in\mathrm{Sym}^{kd}(W^*)$ such that there exist $r+1$ covariants of $\mathrm{Sym}^d(W^*)$ which are linearly independent at $f$, 
which we take to be $q_0,\ldots,q_r$. We can form the matrix $Q_{q(f)^*,q(f)}$, and the matrix $Q_{b^*,b}$, where
$b = (b_i)$ denotes the canonical monomial basis of $\Sym^d(W^*)$ in lexicographic order, and $b^*$
its dual basis. Our aim is to find a change of basis to transform $Q_{q(f)^*,q(f)}$ into $Q_{b^*,b}$, in a way which is given by the next lemma.

\begin{lemma}
    For any $M\in\mathrm{GL}_{r+1}$, we have \[Q_{Mq(f)^*, {}^tM^{-1}q(f)} = (M\otimes M)Q_{q(f)^*, q(f)}(M\otimes M)^{-1}\,.\]
\end{lemma}

Thus, our algorithmic solution is to try and find a matrix $M\in\mathrm{GL}_{r+1}$ such that

\[Q_{b^*, b} = (M\otimes M)Q_{q(f)^*, q(f)}(M\otimes M)^{-1}\,,\]
or equivalently 
\begin{equation}\label{eq:good_basis_pullback}
    Q_{b^*, b}(M\otimes M) = (M\otimes M)Q_{q(f)^*, q(f)}\,.
\end{equation}

We know that there exists a matrix $M$ satisfying these equations, which is the change of basis from the basis $q(f)$ to $b$.
Unfortunately $\varphi(f) = q(f)$ is not known, thus we need to understand how to obtain a solution to Equation~\eqref{eq:good_basis_pullback}.
Let us consider $M = (m_{i,j})$ as a matrix of $(r+1)^2$ indeterminates.

Then by changing $p(f)$ into $Mp(f)$, and $p(f)^*$ into ${}^tM^{-1}p(f)^*$,
the quadratic relations we obtain are the ones corresponding to the canonical embedding.
Equation~\eqref{eq:good_basis_pullback} becomes a system of $(r+1)^4$ quadratic equations in $(r+1)^2$ indeterminates.
In our area of application, we can solve that system by computing a Noether normalization~\cite[Lemma 2.5.7]{derksen}.
This yields a linear combination of variables which are algebraically independent $(m_l)_{l\in L}$, where $L\subseteq\{(i,j)~|~0\leq i,j\leq r\}$.
Moreover, any other variable $m_l$ for $l\notin L$ satisfies an integral relation on $K[m_l~|~l\in L]$.

Therefore we can for instance assign arbitrary values to the indeterminates $m_l$ for $l\in L$, and finding the rest of the indeterminates then amounts to 
taking some field extensions. 

Once a solution is known, we update $\tilde{f}(X)$ to $g := \tilde{f}({}^tMX)$,
and the form \[g(b) = g\left(x_0^d,x_0^{d-1}x_1,x_0^{d-1}x_2,\ldots,x_n^d\right)\in\Sym^{kd}(W^*)\] is $\mathrm{GL}_{n+1}$-equivalent to $f$.

\begin{rem}\label{rem:control_extension}
    The author does not know how to control the size of the field extensions over which the parametrization is computed.
    There might exist an approach to this specific instance of the parametrization problem 
    which would always compute a parametrization over the smallest possible field, but the author is not aware of it.

    In practice, one should try as much as possible to use covariants of small orders to reduce the degrees of the fields extensions.
    
    The most favorable case is when the covariants are of degree $1$. Then $\tilde{f}$ belongs to $\Sym^{kd}(W^*)$, and is directly 
    $\mathrm{GL}_{n+1}$-equivalent to $f$. In addition, the coefficients of $\tilde{f}$ lie in the field where the invariants of $f$ live. 

    Unfortunately sometimes covariants of order $1$ do not exist. For instance, binary forms of even degree only have covariants of degree $2$.
    The image of $\varphi$ is always a conic, and can be parametrized if it has a rational point. Otherwise, the parametrization is defined over a quadratic extension of the field in which the invariants lie.

    For other higher degrees or greater number of variables, there exist several quadratic relations, and the algorithm can be quite impractical.
\end{rem}

\subsection{Main algorithm}

We derive a reconstruction algorithm from Corollary~\ref{cor:general_reconstruction}.

\begin{cor}\label{cor:general}
    Let $k$, $d$, and $n$ be positive integers. Let $K$ be an algebraically closed field of characteristic $0$ or $p > kd$. Let $W$ be a $K$-vector space with basis $w_0,\ldots, w_n$ and dual basis $x_0,\ldots, x_n$, and let $r = \dim_K(\Sym^d(W^*))-1$.
    We assume that there exist $q_0,\ldots, q_r$ covariants of order $d$ which
    are generically linearly independent. 
    Let $(I_j)_{j\in J}$ be a (finite) set of generators of $K[\Sym^{kd}(W^*)]^{\mathrm{SL_{n+1}}}$.
    
    There exists an algorithm, which, given $(I_j(f))_{j\in J}$ corresponding to $f\in\Sym^{kd}(W^*)$ such that $q_0,\ldots,q_r$ are linearly independent at $f$, returns a form $f'\in\Sym^{kd}(W^*)$ with the same invariants as $f$.
\end{cor}

We explain how to derive such an algorithm from Theorem~\ref{thm:general}.
Here is a high-level description of this algorithm:

\begin{enumerate}
    \item The first step (which is done only once for every set of covariants) consists in the precalculation of a decomposition on a generating set of invariants
    of all the invariants required for the computation of $\tilde{f}$, of the matrix $\big(\D(\Delta q_i^*\Delta q_j^*\,, q_lq_m)\big)$,
    and of the invariant $\det(q_0,\ldots,q_r)$. That step can be done using an evaluation-interpolation strategy.

    \item We can then specialize these formulas to an $f$ represented by a list of invariants $(I_j(f))_{j\in J}$ by evaluating the decomposition polynomials at the
    values of the generating set of invariants at $f$. We check the condition of linear independance at $f$ by specializing the invariant $\det(q_0,\ldots,q_r)$ at $f$.

    \item If the determinant is not $0$, we can parametrize $\mathrm{Im}(\varphi)$ from the quadratic forms by using Section~\ref{sec:param}. 
    Then, we evaluate $\tilde{f}$ on this parametrization, and we recover a form $f'\in \Sym^{kd}(W^*)$ which is $\mathrm{GL}_{n+1}$-equivalent to $f$.
\end{enumerate}

\subsection{Improving the algorithm}
\mbox{}

In order to be able to use this algorithm in practice, we need to reduce the degrees of the invariants used.
Indeed, the degrees of the invariants involved in the formulas in Theorem~\ref{thm:general} can be high, rendering our technique very much not effective.

For instance, if we choose $q_1,\ldots,q_r$ generically linearly independent covariants of
$\Sym^{kd}(W^*)$ of respective degrees $d_1,\ldots ,d_r$,
their dual basis $\Delta q_1^*,\ldots,\Delta q_r^*$ of contravariants has high degree: $\deg(\Delta q_i^*) = d_i+2\sum_{j\neq i}d_j$.
Thus the coefficients of $\tilde{f}$ are invariants of very high degrees,
hence they are not effectively decomposable on a generating set of invariants.

To remedy this problem, we can choose a set of covariants $q = (q_i)$ and a set of contravariants $p = (p_i)$ of the same order,
which generically form a basis of their respective spaces.
Then, computing a dual basis $p(f)^*$ of $p(f)$ in terms of $q(f)$ (or the reverse) is just a task of linear algebra, as established in the following lemma.

\begin{lemma}
    Let $M_{p,q} = \left(\D(p_i,q_j)\right)_{i,j}$.
    Then the basis $(p_i(f)^*)_i$ can be expressed using the basis $(q_j(f))_j$. We have in particular
    \[(p_i(f)^*)_i = M_{p,q}^{-1}(q_j(f))_j.\]
    
    Hence, if we let \[\tilde{f}_{p,q} = \sum_{0\leq i_1,\ldots, i_k \leq r}\D(p_{i_1}\cdots p_{i_k}, f)X_{i_1}\cdots X_{i_k},\]
    and \[Q_{p,q} = \big(\D(p_ip_j\,,
    q_lq_m)(f)\big)_{\substack{0\leq i,j\leq r\\0\leq l, m\leq r}},\]
    we have \[\tilde{f}_{p,q}(p_0(f)^*,\ldots, p_r(f)^*) = \frac{(kd)!}{d!^k}f,\]
    and the quadratic relations between the $p_i(f)^*$ can be known by computing a basis of 
    the right kernel of $Q_{p,q}{}^tM_{p,q}^{-1}$.
\end{lemma}

To summarize, here is what need to be precomputed:
\begin{enumerate}
    \item The decomposition of $\D(p_i, q_j)$ for all $0\leq i,j\leq r$ for the computation of the dual basis of $p$ by inverting $M_{p,q}$.
    \item The decomposition of $\D(p_ip_j, q_lq_m)$ for all $0\leq i,j,l,m\leq r$ for the computation of the quadratic forms.
    \item The decomposition of $\D(p_{i_1}\cdots p_{i_k}, f)$ for all $0\leq i_1,\ldots,i_k\leq r$ for the computation of $\tilde{f}$.
\end{enumerate}
It consists in a total of $(r+1)^k+(r+1)^4+(r+1)^2$ invariants, of degrees at most $\max(d_p^2d_q^2\,, d_p^k+1)$,
where $d_p$ (resp. $d_q$) is the maximal degree of the contravariants (resp. covariants).
\medskip

\begin{rem}\label{rem:order1}
    When $d = 1$, we have $r=n$, so the morphism $\varphi$ introduced
    in Section~\ref{sec:theory} is just an automorphism of
    $\mathbb{P}^n$.
    Hence, by choosing $n+1$ contravariants of order $d$ which are linearly independent at $f$, we obtain
    \[\tilde{f} = \sum_{0\leq i_1,\ldots, i_k\leq r}\D(p_{i_1}(f)\cdots p_{i_k}(f), f)X_{i_1}\cdots X_{i_k}\,.\]
    Since \[\tilde{f}(p_0(f)^*,\ldots, p_r(f)^*) = \frac{(kd)!}{d!^k}f\,,\] it is clear that $\tilde{f}$ and $f$ are $\mathrm{GL}_{n+1}$-equivalent.
    Moreover, the field over which $\tilde{f}$ is defined is the field in which the invariants lie.
\end{rem}

Building on that remark, we provide a new reconstruction algorithm for smooth plane quartics in Section~\ref{sec:reconstruction_genus3}.

\begin{rem}
    We note that in the opposite case $k = 1$, the situation is significantly worse: eventhough the form $f$ itself is a covariant, the number of 
    invariants needed explodes, and the step of parametrization of $\mathrm{Im}(\varphi)$ becomes unmanageable. 
    As a result, in practice, we use $d=1,2$ whenever possible.
\end{rem}

\section{Reconstruction of smooth hypersurfaces}
\label{sec:smooth}

Let $K$ be an algebraically closed field of characteristic $0$.
Let $k$, $d$, and $n$ be positive integers. Let $W$ be a $(n+1)$-dimensional $K$-vector space, and let $W^*$ denote its dual.
In this paragraph, we show that under mild assumptions on $f\in\Sym^{kd}(W^*)$,
there exist $\dim_K(\Sym^d(W^*))$ covariants of $\Sym^{kd}(W^*)$ which are linearly independent at $f$.
We use the notion of stability defined in Mumford's GIT~\citep{mumford}. One can find an exposure that suits our needs in~\cite[Chapters 8,9]{dolgachev}.
\medskip

Let us recall an important result.

\begin{prop}[\protect{\cite[Prop 3.1]{domokos}}]\label{prop:domokos}
    Let $G$ be a linearly reductive group, $X$ an affine $G$-variety, and $W$ a $G$-module.
    If for some $x\in X$ having closed orbit the stabilizer $G_x$ acts trivially on $W$,
    then there exist $s = \dim_K(W)$ covariants $F_1, \ldots, F_s\in \mathrm{Cov}_G(X, W)$ such that $F_1(x), \ldots, F_s(x)$ are linearly independent over $K$.
\end{prop}

\begin{rem}
    The linearly reductive condition on $G$ implies that we work over a field of characteristic $0$, since $\mathrm{GL}_{n+1}$ and $\mathrm{SL}_{n+1}$ 
    are not linearly reductive in positive characteristic~\cite[Theorem 2.2.19]{derksen}. The authors wonders whether the linear reductivity condition can be weakened to work in positive characteristic as well,
    for example by using the theory of good filtrations~\citep{jantzen,derksen_weyl}.
\end{rem}

\begin{prop}\label{prop:covs}
    For every stable $f\in \Sym^{kd}(W^*)$ with trivial reduced stabilizer, the following statements are equivalent:
    \begin{enumerate}
        \item There exist $q_0,\ldots, q_r$ covariants of order $d$ which are linearly independent at $f$,
        \item $\displaystyle\gcd\left (k, \frac{(n+1)}{\gcd(n+1, d)}\right ) = 1$. 
    \end{enumerate}
\end{prop}

\begin{proof}
    Since the existence of covariants of order $d$ of $\Sym^{kd}(W^*)$ implies the second statement, we only need to prove the converse.

    We apply Proposition~\ref{prop:domokos} with $G = \mathrm{SL_{n+1}}$, $X=\Sym^{kd}(W^*)$, and $W = \Sym^d(W^*)$.
    Since $\mathrm{char}(K) = 0$, we know that $G$ is linearly reductive~\cite{haboush}.

    It is known that the elements of the stable locus have closed orbit~\cite[Chapter 8]{dolgachev}.
    Now, if let $f\in X$ with trivial reduced stabilizer, we have \[G_f = \{\lambda \Id~\vert~\lambda^{kd}=\lambda^{n+1} = 1\}.\]

    We derive from this equality that 
    $G_f$ acts trivially on $W = \Sym^d(W^*)$ if and only if for all $\lambda\in G_f$, we have $\lambda^d = 1$.
    This is the case if and only if $\gcd(kd, n+1) = \gcd(d, n+1)$, and this condition can be rewritten as \[\gcd\left (k, \displaystyle\frac{n+1}{\gcd(d, n+1)}\right ) = 1\,.\]

    In addition, we have \[\frac{kd\alpha-d}{n+1} = \frac{d}{\gcd(d ,n+1)}\cdot \frac{k\alpha -1}{\displaystyle\frac{n+1}{\gcd(d, n+1)}}\,.\]
    In other words, for any order $d$ for which covariants of $\Sym^{kd}(W^*)$ might exist (meaning for which the weight $\frac{kd\alpha-d}{n+1}$ is an integer),
    there must exist at least $\dim_K(\Sym^{d}(W^*))$ generically linearly independent covariants of $\Sym^{kd}(W^*)$.

    In that case, we apply Proposition~\ref{prop:domokos}, which implies the existence of the desired covariants.
\end{proof}

\begin{cor}\label{cor:smooth_hyp}
    Let us assume that $kd\geq 3$, and that $\displaystyle\gcd\left (k, \displaystyle\frac{(n+1)}{\gcd(n+1, d)}\right ) = 1$. 
    For any form $f\in \Sym^{kd}(W^*)$ such that which defines a smooth hypersurface with trivial automorphism group, the reconstruction algorithm (Corollary~\ref{cor:general}) applies.
\end{cor}

\begin{proof}
    According to~\cite[Theorem 10.1]{dolgachev}, any non-singular element of $\Sym^{kd}(W^*)$ is stable.
    Hence, by Proposition~\ref{prop:covs}, there exist covariants that can be used to meet the requirements of Corollary~\ref{cor:general}.
\end{proof}

\begin{rem}\label{rem:generating_cov}
    Let $C_d$ be the $K[\Sym^{kd}(W^*)]^{\mathrm{SL}_{n+1}}$-module of covariants of order $d$.
    Since the space of covariants is finitely generated, so is $C_d$. Let $q_0,\ldots, q_l$ be a generating family of $C_d$.
    Now let $f\in\Sym^{kd}(W^*)$ be stable such that $f$ has trivial reduced stabilizer.
    Then there exist $r = \dim_K(\Sym^d(W^*))$ covariants of order $d$ which are linearly independent at $f$.
    Hence, there must be $r$ covariants in the set $q_0,\ldots, q_l$ which are linearly independent at $f$.

    Therefore, if we have at our disposal a generating set of covariants of a given degree,
    then we know that any stable $f\in\Sym^{kd}(W^*)$ such that $f$ has trivial reduced stabilizer can be covariantly reconstructed by using only
    a subset of these generating covariants.
    Unfortunately, determining such a set of covariants is usually out of reach.
\end{rem}

\section{Examples}
\label{sec:examples}

\subsection{Binary forms}
\mbox{}

We turn to the case of binary forms, for which reconstruction algorithms have been found by~\cite{mestre} and~\cite{noordsij}.
Let $W$ be a $2$-dimensional $K$-vector space with basis $w_0,w_1$ and dual basis $x_0,x_1$. They used bases of covariants of oreder $2$ and $1$ respectively,
which enabled them to solve the reconstruction problem for binary forms of even degrees and odd degrees respectively.

\begin{definition}
    Let $f\in \Sym^d(W^*), g\in \Sym^e(W^*)$ for some positive integers $d$ and $e$.
    For any positive integer $l$, we define the \textit{transvectant} of level $l$ of $f$ and $g$ to be
    \[(f,g)_l = \sum_{i = 0}^l(-1)^i\binom{l}{i}\frac{\partial^l f}{\partial^ix_0\partial^{l-i}x_1}\frac{\partial^l g}{\partial^{l-i}x_0\partial^ix_1}\in Sym^{d+e-2l}(W^*).\]
\end{definition}

\begin{prop}
    Let $r$ be a positive integer, and let $K$ be an algebraically closed field of characteristic $0$ or $p> r$. If we define the linear function 
    \[\function{\tau_r}{\mathrm{Sym}^r(W^*)}{\mathrm{Sym}^r(W)}{x_0^ix_1^j}{r!(-1)^iw_0^jw_1^i},\]
    then for all $C\in\Sym^r(W^*), C'\in\Sym^{r'}(W^*)$, we have:
        \[\D(\tau_r(C), C') = (C, C')_{r}.\]
\end{prop}

\begin{rem}
    We note that if $C$ is a covariant of $\Sym^d(W^*)$ of order $r$, then $\tau_r(C)$ is a contravariant of the same space.
    Thus for binary forms, the notions of covariants and contravariants are essentially the same notion, and we typically speak only in terms of covariants.

    However, the theory of covariants and contravariants is distinct for polynomials in more than $2$ variables, thus the map $\tau_r$ does not generalize.
\end{rem}

A similar statement holds for its inverse map $\tau_r^{-1}$, which maps contravariants to covariants.
These functions make the connection between the transvectant operator for binary forms and the operator $\D$.
Hence, if $(q_i)_i$ is a family of covariants of order $d$ of $\Sym^{kd}(W^*)$ which are generically linearly independent,
then $(p_i := \tau(q_i))_i$ is a family of contravariants of order $d$ of $\Sym^{kd}(W^*)$ which are generically linearly independent.
Thus, one can use the families $p_i$ and $q_j$ to reconstruct a generic element of $\Sym^{kd}(W^*)$.

\begin{lemma}
    We assume that $K$ is of characteristic $0$ or $p > kd$, and let $q_1,\ldots, q_k\in \mathrm{Sym}^d(W^*)$.
    Then we have 

    \begin{equation}\label{eq:multiplicativity_tau}
        \tau_d(q_1)\cdots \tau_d(q_k) = \frac{d!^k}{(kd)!}\tau_{kd}(q_1\cdots q_k).
    \end{equation}
\end{lemma}

We now detail the cases $d=1,2$.

For odd $k$, the condition on the $\gcd$ of Proposition~\ref{prop:covs} can be satisfied with $d=1$.
\begin{cor}\label{cor:odd_bin}
    Let $f$ be a stable binary form of odd degree $k\geq 5$ with trivial reduced stabilizer, which is generically the case. Then there exist covariants $q_0$ and $q_1$ 
    of order $1$ which are linearly independent at $f$.
    If we let \[\tilde{f} = \sum_{i = 0}^k\binom{k}{i}\Big(q_0^i(f) q_1^{k-i}(f), f\Big)_k X_0^i X_1^{k-i},\]
    then $\tilde{f}$ is $\mathrm{GL}_2$-equivalent to $f$.
    Moreover, the coefficients of $\tilde{f}$ lie in the base field of the invariants.
\end{cor}

\begin{proof}
    This result is a corollary of Theorem~\ref{thm:general}.
    Indeed, if $p_0(f)^*,\ldots, p_r(f)^*$ is the dual basis of $p_0(f),\ldots, p_r(f)$, then $\tilde{f}(p_0(f)^*,\ldots, p_r(f)^*) = f$.
    We observe that the constant $\frac{(kd)!}{d!^k}$ of Equation~\eqref{eq:decomposition_f} cancels with the constant $\frac{d!^k}{(kd)!}$ of Equation~\eqref{eq:multiplicativity_tau}.
\end{proof}

\begin{rem}
    This statement is established in~\cite[Theorem 3.10]{noordsij}, except for the coefficients of $\tilde{f}$, which are not written with transvectants. 
    Moreover, binary forms of degree $5$ with automorphisms are covered in~\cite{noordsij}, as well as positive characteristic. These cases are not treated here.
\end{rem}

However, for binary forms of even degree $k$, we have $\displaystyle\gcd\left (k, \frac{2}{\gcd(1, 2)}\right ) = 2\ne 1$.
Thus, binary forms of even degree cannot have covariants of order $1$, so we must turn to covariants of order $2$.

\begin{cor}\label{cor:even_bin}
    Let $f$ be a stable binary form of even degree $k\geq 6$ with trivial reduced stabilizer, which is generically the case. Then there exist 3 covariants $q_0,q_1,q_2$ of order 2 
    which are linearly independent at $f$.
    If we let \begin{align*}
        \tilde{f} &= \sum_{\substack{0\leq i, j\leq k\\i+j \leq k}}\binom{k}{i,j}\Big(q_0^i(f) q_1^j(f) q_2^{k-i-j}(f), f\Big)_{2k} X_0^i X_1^j X_2^{k-i-j},\text{ and}\\
        Q &= \sum_{0\leq i, j\leq 2} (q_i(f), q_j(f))_2X_i X_j,
    \end{align*}
    then one can recover $f'\in\Sym^{2k}(W^*)$ from $Q$ and $\tilde{f}$, such that $f'$ is $\mathrm{GL}_2$-equivalent to $f$.
    Moreover, the coefficients of $f'$ lie in at most a quadratic extension of the base field of the invariants, depending on whether the conic defined by $Q$ has a rational point.
\end{cor}

\begin{rem}
    This is essentially Mestre's approach~\citep{mestre}. He argues from Clebsch's formulas~\citep{clebsch} that the dual basis $(q_0(f)^*,q_1(f)^*,q_2(f)^*)$ must satisfy the quadratic relation $Q = 0$. 
    Then, by finding a point on the conic $Q = 0$, he parametrizes it, and by reinjecting in
    $\tilde{f}$, he obtains an element $f'\in\Sym^k(W^*)$ which is $\mathrm{GL}_2$-equivalent to $f$.        
\end{rem}

Our method extends the existing reconstruction algorithms to direct sums of binary spaces.

\begin{prop}\label{prop:sum_bin}
    Let $s>1$, $k_1,\ldots, k_s > 0$ be integers, and let $d = 1$ or $2$ such that if $2 \vert\gcd(k_1,\ldots,k_s)$, then $d = 2$.
    Let $K$ be an algebraically closed field of characteristic $0$
    or $p > d\max(k_i)$. Let $W$ be a $2$-dimensional $K$-vector space. Let \[W' = \Sym^{dk_1}(W^*)\oplus \ldots \oplus \Sym^{dk_s}(W^*),\]
    and let $f = (f_1,\ldots,f_s)\in W'$ such that $f$ is stable in $W'$, and with trivial reduced stabilizer.
    There are $2$ cases:

    \begin{enumerate}
        \item If $d = 2$, then there exist $3$ covariants of order $2$ of $W'$ which are linearly independent at $f$, where a covariant here means a $\mathrm{SL}_2$-equivariant map $W'\rightarrow \Sym^r(W^*)$ for some nonnegative integer $r$.
        Let $q_0$, $q_1$, and $q_2$ be such covariants. For all $1\leq i\leq s$, we let
        \[\tilde{f_i} = \sum_{\substack{0\leq l,m\leq k_i\\l+m \leq k_i}}\binom{k_i}{l,m}\Big(q_0^l(f) q_1^m(f) q_2^{k_i-l-m}(f), f\Big)_{2k_i} X_0^l X_1^m X_2^{k_i-l-m}\,,\]
        and \[Q = \sum_{0\leq i,j\leq 2} (q_l(f), q_m(f))_2X_l X_m.\]
        Then the coefficients of $Q$ and the $\tilde{f_i}$'s are invariants of $W'$, and we can recover $f' = (f_1',\ldots,f_s')\in W'$ which is $\mathrm{GL}_2$-equivalent to $f$
        only from the data of $Q$ and all the $\tilde{f_i}$.
        
        \item If $d = 1$, then there exist $2$ covariants of order $1$ of $W'$ which are linearly independent at $f$, which we take to be $q_0$ and $q_1$.
        For all $1\leq i\leq s$, we let
        \[\tilde{f_i} = \sum_{0\leq l\leq k_i}\binom{k_i}{l}\Big(q_0^l(f) q_1^{k_i-l}(f), f\Big)_{k_i} X_0^l X_1^{k_i-l}\,.\]
        Then the coefficients of the $\tilde{f_i}$'s are invariants of $W'$, and $f' = (\tilde{f_1},\ldots,\tilde{f_s})\in W'$ is $\mathrm{GL}_2$-equivalent to $f$.
    \end{enumerate}
\end{prop}

With Proposition~\ref{prop:sum_bin}, and Corollaries~\ref{cor:odd_bin} and~\ref{cor:even_bin}, we derive a reconstruction algorithm for direct sums of binary spaces.
The author is not aware of the existence of such an algorithm in the litterature. Until now, the reconstruction algorithms of direct sums of binary spaces
first reconstructed a form of highest degree $\max(dk_i)$, and were able to reconstruct the other forms using Gröbner bases and the mixed conditions on the invariants (see e.g.~\cite{LRS}). 

\begin{ex}\label{ex:g4}
    Let $W$ be a $2$-dimensional $K$-vector space, and define \[W' := \Sym^6(W^*)\oplus \Sym^4(W^*).\]
    
    We pick $3$ covariants of $W$ of order $2$ which are generically linearly independent:
    \begin{align*}
        q_0(f) &= (f_6, f_4)_4\,,\\
        q_1(f) &= (f_6, f_4^2)_6\,,\text{and }\\
        q_2(f) &= (f_6^2, f_4^3)_{11}\,,
    \end{align*}
    where $f = (f_6, f_4)\in W'$.
    We define 
    \begin{align*}
        \tilde{f_6} &= \sum_{\substack{0\leq i, j, k \leq 2}}(q_i(f) q_j(f) q_k(f), f_6)_6 X_iX_jX_k\\
        \tilde{f_4} &= \sum_{\substack{0\leq i, j\leq 2\\i+j \leq 2}}(q_i(f)q_j(f), f_4)_4 X_iX_j\\
        Q &= \sum_{0\leq i, j\leq k} (q_i(f), q_j(f))_2X_i X_j,
    \end{align*}

    and we let $\varphi$ denote the Veronese embedding
    \[[x:y]\longmapsto [q_0(f):q_1(f):q_2(f)].\] 

    Its image is defined by $Q$, and we know that
    \begin{align*}
        f_6 &= \tilde{f_6}(\tau_2(q_0(f))^*,\tau_2(q_1(f))^*,\tau_2(q_2(f))^*)\\
        f_4 &= \tilde{f_4}(\tau_2(q_0(f))^*,\tau_2(q_1(f))^*,\tau_2(q_2(f))^*).
    \end{align*}
    
    Finding a point on the conic defined by $Q$ allows us to parametrize it. The evaluation of $\tilde{f_6}$ and $\tilde{f_4}$ on this parametrization
    gives $f'=(f_6',f_4')\in W'$, which is $\mathrm{GL}_2$-equivalent to $f$.
\end{ex}

\begin{rem}
    Olive proved that a minimal set of generating covariants of order $2$
    of $W'$ (as a $K[W']^{\mathrm{SL}_2}$-module) is generated by $68$ elements~\cite[Theorem 8.3]{olive-gordan}.
    Hence, if $f$ belongs to the stable locus of $W'$ and has trivial reduced stabilizer, there exist $3$ covariants of order $2$ of $W'$
    which are linearly independent at $f$ by Proposition~\ref{prop:domokos}. These covariants can be taken in the generating set of $68$ covariants, by Remark~\ref{rem:generating_cov}.
\end{rem}

\subsection{Reconstruction of non-hyperelliptic curves of genus $3$}
\label{sec:reconstruction_genus3}
\mbox{}

The canonical embedding of a non-hyperelliptic curve of genus 3 is given by a smooth, irreducible plane quartic, i.e. defined by a ternary form of degree $4$.
The isomorphism classes of these curves are completely determined by the 13 Dixmier-Ohno invariants~\cite{dixmier, ohno}.
In~\cite{LRS}, the authors give an algorithm to reconstruct a generic plane quartic from the data of the Dixmier-Ohno invariants.
They use an exceptional isomorphism between $\mathrm{SO}_3$ and $\mathrm{SL}_2/\{\pm{1}\}$ to reduce to the known case of binary forms.

Their algorithm involves a construction over a quadratic extension of the field of definition of the invariants.
In addition, the authors make the generic assumption that $I_{12}\ne 0$.
\medskip

We present an algorithm that theoretically solves the problem of reconstruction of plane quartics from the Dixmier-Ohno invariants
in more generality. Indeed, the set of smooth plane quartics with non-trivial automorphism group is of codimension 2, compared to codimension 1 for the hypersurface defined by $I_{12}$.
    
\begin{theorem}
    Let $W$ be a $3$-dimensional $K$-vector space, and let $f\in\Sym^4(W^*)$ such that $f$ is stable and has trivial reduced stabilizer.
    Then there exist 3 contravariants of order $1$ of $\Sym^4(W^*)$, which are linearly independent at $f$. 
    Let $p_0$, $p_1$, and $p_2$ be such contravariants. Then
    \[\tilde{f} = \sum_{0\leq i_1,\ldots, i_4\leq 2}\D(p_{i_1}\cdots p_{i_4}, f)X_{i_1}\cdots X_{i_4}\]
is $\mathrm{GL}_{3}$-equivalent to $f$.
\end{theorem}

The transvectant of ternary forms is defined as the determinant of the $\Omega$-process~\citep{olver} for ternary forms. It takes $3$ arguments,
and is denoted $(\cdot,\cdot,\cdot)_l$, where $l$ is a nonnegative integer. Let $'$ be the operator defined in~\citep[End of page 6]{kohel}.
This operator allows to change covariants into contravariants, and vice-versa.
We now construct contravariants $p_0$, $p_1$, and $p_2$ by considering the covariants and contravariants in Table~\ref{tab:def_covcont_genus3}.

By Remark~\ref{rem:order1}, 
\[\tilde{f} = \sum_{0\leq i_1,\ldots, i_4\leq 2}\D(p_{i_1}(f)\cdots p_{i_4}(f), f)X_{i_1}\cdots X_{i_4}\]
is $\mathrm{GL}_{3}$-equivalent to $f$.
\medskip

\begin{rem}
    By Remark~\ref{rem:generating_cov}, finding a generating set of order $1$ contravariants of
    $\Sym^{4}(W^*)$ is enough to reconstruct all smooth non-hyperelliptic curves of genus 3 with no automorphisms.
    Presently, the author does not know such a generating set.
    However, we give $3$ contravariants of order $1$ which are generically linearly independent,
    and allow to reconstruct generically, except on a hypersurface given by the vanishing of the determinant of the three contravariants.
\end{rem}

For the precomputation phase, we need the decomposition of the invariants $\D(p_{i_1}\cdots p_{i_4}, \Id)$
for all $0\leq i_1\leq\ldots\leq i_4\leq 2$, for a total of $15$ invariants.
The degrees of the invariants vary from $57$ to $69$, and their decomposition (calculated using a method of evaluation-interpolation)
took at most $1$ day of computation. These decompositions are stored in~\cite[\protect{\texttt{Decomposition\_genus3.m}}]{git-reconstruction}.
\medskip

After this precomputation step, the actual reconstruction algorithm which takes a list of Dixmier-Ohno invariants and returns a ternary quartic takes around 0.2 seconds in practice for reasonably sized entries.

Let us illustrate the computation with an example. For clarity, we first compute the contravariants and then derive $\tilde{f}$ even though, the user does not have access to the contravariants.
In practice, since the coefficients of $\tilde{f}$ are known polynomials in the Dixmier-Ohno invariants, they can be directly evaluated from the invariants of a given example.

\begin{ex}
    Let \begin{align*}
        f = &-745x_0^3x_2 - 6705x_0^2x_1x_2 - 75990x_0^2x_2^2 - 1788x_0x_1^3 - 36207x_0x_1^2x_2 - 571266x_0x_1x_2^2\\
            &- 1827336x_0x_2^3 - 7152x_1^4 - 123819x_1^3x_2 - 1834488x_1^2x_2^2 + 950004x_1x_2^3 - 631522x_2^4
    \end{align*}
    be a ternary quartic form whose Dixmier-Ohno invariant $I_{12}$ is $0$ (this case is not covered by the existing algorithm of~\cite{LRS}). 
    It defines a smooth non-hyperelliptic curve of genus $3$, since $I_{27}\ne 0$.
    This equation was established using a work of Shioda~\cite{shioda}.   
    
    We compute its contravariants $p_0$, $p_1$, and $p_2$ as in Table~\ref{tab:def_covcont_genus3}.
    Up to scaling, we find
    \begin{align*}
        p_0 = &-36028900960739935302662w_0 + 2546868783781471003910w_1\\ &-
        207634621252481717745w_2,\\
        p_1 = &-167266167826007043607549539758w_0 + 11957094310556682023883659540w_1\\&-
        996728625589442333471190105w_2,\\
        p_2 = &-2137425487531362504044770w_0 + 192739452116090004098632w_1\\ &-
        4823065036939209106179w_2.
    \end{align*}

    We compute $\tilde{f}$:
    its expression is too large to be displayed here, but its coefficient in $x_0^4$ is
    {\footnotesize
    \[-151647765305065905238548582432828758523321832584926229590543175552953534711319971363994226800.\]
    }
    
    The other coefficients have similar sizes. The coefficients of this model can be reduced: we use Elsenhans and Stoll's minimization algorithm of ternary forms~\cite{elsenhans},
    and obtain the minimized model
    \begin{align*}
        f' =&\ 1428254x_0^4 + 1615140x_0^3x_1 - 747384x_0^3x_2 + 1802304x_0^2x_1^2 + 222606x_0^2x_1x_2 + 4470x_0^2x_2^2\\
        & + 1489404x_0x_1^3 + 337932x_0x_1^2x_2 + 26820x_0x_1x_2^2 + 745x_0x_2^3 + 19668x_1^4 + 1788x_1^3x_2.
    \end{align*}

    A simple computation shows that the Dixmier-Ohno invariants of $f'$ are the same as those of $f$.
    Hence $f'$ and $f$ are $\mathrm{GL}_3$-equivalent. 
\end{ex}

\subsection{Reconstruction of non-hyperelliptic curves of genus $4$}
\label{sec:reconstruction_genus4}
\mbox{}

Let $K$ be an algebraically closed field of characteristic $0$. 
Let $\mathcal{C}$ be the canonical embedding in $\mathbb{P}^3$ of a (smooth, irreducible) non-hyperelliptic curve of genus 4 defined over $K$.
Then $\mathcal{C}$ is the complete intersection of a quadric and a cubic.
Let $Q, E\in K[X,Y,Z,T]$ be homogeneous irreducible forms of degree $2$ and $3$ respectively, which define $\mathcal{C}$.

Since $Q$ is irreducible, it must be of rank $3$ or $4$ as a quadratic form. The case of rank $3$ reduces to the reconstruction of elements from $\Sym^6(W^*)\oplus \Sym^4(W^*)$ (for more details, we refer the reader to~\cite{bouchet}).
In fact, any smooth non-hyperelliptic genus $4$ curve $C$ lying on a singular quadric
can be defined as the vanishing locus of $F = w^3+wf_4(s,t)+f_6(s,t)$ in the weighted space $\mathbb{P}(1,1,2)$, where $w$ is of weight $2$.

\begin{prop}[Example~\ref{ex:g4}]
    Let $W$ be a $2$-dimensional $K$-vector space. Let $W' := \Sym^6(W^*)\oplus \Sym^4(W^*)$,
    and let $f = (f_6, f_4)\in W'$.
    
    We pick $3$ covariants of $W$ of order $2$, and assume that they are linearly independent at $f$, which is generically the case:

    \begin{align*}
        q_0(f) &= (f_6, f_4)_4\,,\\
        q_1(f) &= (f_6, f_4^2)_6\,,\text{ and}\\
        q_2(f) &= (f_6^2, f_4^3)_{11}\,.
    \end{align*}

    If we define
    \begin{align*}
        Q &= \sum_{0\leq i,\,j\leq 2}(q_i(f),\,q_j(f))_2X_iX_j,\\
        E &= X_3^3+X_3\left(\sum_{0\leq i,\,j\leq 2}(q_i(f)q_j(f),\,f_4)_4X_iX_j\right)+\sum_{0\leq i,\,j,\,k\leq 2}(q_i(f)q_j(f)q_k(f),\,f_6)_6X_iX_jX_k, 
    \end{align*}
    then the vanishing locus of $Q$ and $E$ is a non-hyperelliptic genus $4$ curve of rank $3$ which is isomorphic to $F$.
\end{prop}

We now treat the generic case, which is the case of rank $4$.
Without loss of generality, we can assume that $Q$ is in normal form $Q = XT-YZ$, which comes from the fact that $Q$ and $XT-YZ$, as quadratic forms, are both of rank $4$, and therefore are $\mathrm{GL}_4$-equivalent. 
Let \[\psi : \mathbb{P}^1\times\mathbb{P}^1\longrightarrow \mathbb{P}^3\] be the
Segre embedding, defined by $\psi([x:y],[u:v]) = [xu:xv:yu:yv]$. The pullback
of the cubic form $E$ via $\psi$ is $E(xu,xv,yu,yv)$, which is a bicubic form in the variables $x,y$ and $u,v$ that we call $f$. 

In a previous article~\citep{bouchet}, the author proved that two bicubic forms define geometrically isomorphic curves
if and only if they are $\mathrm{GL}_2\times \mathrm{GL}_2\rtimes \mathbb{Z}/2\mathbb{Z}$-equivalent,
where the groups $\mathrm{GL}_2$ act on their respective sets of variables,
and $\mathbb{Z}/2\mathbb{Z}$ exchanges them.

\begin{definition}
    Let $r_1, r_2\geq 1$ and $l_1,l_2\geq 0$ be integers. Let $W$ be a $2$-dimension $K$-vector space.
    We define a covariant of $\Sym^{r_1}(W^*)\otimes \Sym^{r_2}(W^*)$ to be a $\mathrm{SL}_2\times\mathrm{SL}_2$-equivariant
    homogeneous polynomial map
    \[C : \Sym^{r_1}(W^*)\otimes \Sym^{r_2}(W^*) \rightarrow \Sym^{l_1}(W^*)\otimes \Sym^{l_2}(W^*)\]
    We call $(l_1,l_2)$ the bi-order of $C$, and $d$ its degree as a homogeneous polynomial map. As before, in the case $l_1=l_2=0$, $C$ is called an invariant.
\end{definition}

There exist covariants of $\Sym^{3}(W^*)\otimes \Sym^{3}(W^*)$ of bi-order $(1,1)$,
which we can define using a transvectant~\cite[Proposition 4]{bouchet}.

We shall denote the transvectant of bi-level $(l,m)$ by $(f, g)_{l,m}$, or even $(f, g)_l$ when $l=m$.
Let $\D_2$ denote the differential operator defined in Section~\ref{sec:generalization} for $s = 2$.
Like in the case of binary forms, there is a link between $\D_2$ and the transvectant.

\begin{prop}
    Let us assume that $\mathrm{char}(K)$ is either $0$ or $p > \max(d,e)$. Let $\tau_{d,e}$ be the linear function defined by \[\function{\tau_{d,e}}{\Sym^d(W^*)\otimes\Sym^e(W^*)}{\Sym^d(W)\otimes\Sym^e(W)}{x^iy^ju^lv^m}{(-1)^{i+l}d!e!\, x^jy^iu^mv^l}\,.\]
    Then for any, $C\in\Sym^{d_1}(W^*)\otimes\Sym^{e_1}(W^*)$ and $C'\in\Sym^{d_2}(W^*)\otimes\Sym^{e_2}(W^*)$, we have:
    \[\D_2(\tau_{d_1,e_1}(C), C') = (C, C')_{d_1,e_1}\,.\]
    Moreover, if $C$ is a covariant of $\Sym^{l_1}(W^*)\otimes\Sym^{l_2}(W^*)$, then $\tau_{d_1,e_1}(C)$ is a contravariant of the same space.
\end{prop}

Hence, in characteristic $0$ or $p> \max(d,e)$, the functions $\tau_{d,e}$ and $\tau_{d,e}^{-1}$ establish the connection between the transvectant operator and the operator $\D_2$, as well as between covariants and contravariants in the context of 
double binary forms.
In the spirit of Mestre, we choose to speak only of covariants.

\begin{theorem}\label{thm:genus4}
    Let $f\in \Sym^{3}(W^*)\otimes \Sym^{3}(W^*)$ be a stable form, with trivial reduced stabilizer (in $\mathrm{GL}_2\times\mathrm{GL}_2\rtimes \mathbb{Z}/2\mathbb{Z}$). 
    There exist 4 covariants of $\Sym^{3}(W^*)\otimes \Sym^{3}(W^*)$ of bi-order $(1,1)$
    which are linearly independent at $f$. Let us denote them by $q_0$, $q_1$, $q_2$ and $q_3$.

    Now let us define \begin{align}\label{eq:reconstruction_g4}
        Q(f) &= \sum_{0\leq i, j\leq 3}(q_i(f), q_j(f))_1X_iX_j,\text{ and}\\
        E(f) &= \sum_{0\leq i, j, l \leq 3} (q_i(f)q_j(f)q_l(f), f)_3X_iX_jX_l.\\
    \end{align}
    Then the genus $4$ curve defined by $Q(f)$ and $E(f)$ is isomorphic to the genus $4$ curve defined by the bicubic form $f$.
    Moreover, the coefficients of $Q$ and $E$ are invariants for the action of $\mathrm{SL}_2\times \mathrm{SL}_2\rtimes \mathbb{Z}/2\mathbb{Z}$.
\end{theorem}

\begin{proof}
    The first part of the statement is an application of~\cite[Prop 3.1]{domokos}:
    $\mathrm{SL}_2\times\mathrm{SL}_2$ is a linearly reductive group, $\Sym^{3}(W^*)\otimes \Sym^{3}(W^*)$ and $\Sym^{1}(W^*)\otimes \Sym^{1}(W^*)$ are
    irreducible $\mathrm{SL}_2\times\mathrm{SL}_2$-modules.
    With a proof similar to Proposition~\ref{prop:covs}, we obtain the existence of linearly independent covariants $q_0$, $q_1$, $q_2$ and $q_3$.
    It is easy to see that these covariants are also $\mathbb{Z}/2\mathbb{Z}$-equivariant.

    The second part is similar to Theorem~\ref{thm:general}, but written with the transvectant instead of the apolar pairing.
    In fact, the statement for $\D_2$ is treated in Section~\ref{sec:generalization}, and we obtain that \[E(\tau_{1,1}(q_0(f))^*,\tau_{1,1}(q_1(f))^*,\tau_{1,1}(q_2(f))^*,\tau_{1,1}(q_3(f))^*) = f.\]

    Morevore, we know that \[\dim(\Sym^{2}(W^*)\otimes \Sym^{2}(W^*)) = 9\,,\] and \[\dim(\Sym^2(\Sym^{1}(W^*)\otimes \Sym^{1}(W^*))) = 10\,.\]
    Hence there is exactly one quadratic relation, up to scaling, between $\tau_{1,1}(q_0(f))^*$, $\tau_{1,1}(q_1(f))^*$, $\tau_{1,1}(q_2(f))^*$ and $\tau_{1,1}(q_3(f))^*$.
    It is easy to check that \[Q(\tau_{1,1}(q_0(f))^*,\ldots,\tau_{1,1}(q_3(f))^*) = 0,\] thus the quadratic relation
    must be given by the quadratic form $Q$.
    
    We conclude that the morphism $\mathbb{P}^1\times\mathbb{P}^1\longrightarrow\mathbb{P}^3$, which sends $[x:y],[u:v]$ to \[[\tau_{1,1}(q_0(f))^*:\tau_{1,1}(q_1(f))^*:\tau_{1,1}(q_2(f))^*:\tau_{1,1}(q_3(f))^*],\]
    is an isomorphism from $\mathbb{P}^1\times\mathbb{P}^1$ to the vanishing locus of $Q$. 
    
    It is possible to find such an isomorphism by putting $Q$ in normal form $XT-YZ$ for example.
    Then, we pullback $E$ via the Segre isomorphism, and the bicubic form obtained is $\mathrm{GL}_2\times\mathrm{GL}_2\rtimes \mathbb{Z}/2\mathbb{Z}$-equivalent to $f$.
    As a consequence, the genus $4$ curve defined by $Q(f)$ and $E(f)$ is isomorphic to the genus $4$ curve defined by the bicubic form $f$.
    
    Finally, the coefficients of $Q(f)$ and $E(f)$ are specializations of invariants of $\Sym^{3}(W^*)\otimes \Sym^{3}(W^*)$ for the action of $\mathrm{SL}_2\times \mathrm{SL}_2\rtimes \mathbb{Z}/2\mathbb{Z}$,
    which concludes the proof.
\end{proof}

\medskip
By Remark~\ref{rem:generating_cov}, finding a generating set of bi-order $(1,1)$ covariants of
$\Sym^{3}(W^*)\otimes \Sym^{3}(W^*)$ is enough to reconstruct all smooth non-hyperelliptic curves of genus 4 and rank 4 with no automorphisms.
Presently, the author does not know such a generating set.

However, we give in Table~\ref{fig:tab_cov} a set of $4$ covariants which allow to reconstruct generically.
Other potential covariants can be found in~\cite[Table 1]{bouchet}.

The covariants $c_{31}$, $c_{51,1}$, $c_{51,2}$, and $c_{51,3}$ of Table~\ref{fig:tab_cov} are generically linearly independent.
The degrees of the invariants of $\Sym^{3}(W^*)\otimes \Sym^{3}(W^*)$ involved range between $6$ and $16$.
These invariants have a very nice decomposition on the basis of $65$ invariants, as all but one of these invariants are already inclused in the basis chosen by the author.
Hence the reconstruction algorithm for non-hyperelliptic curves of genus $4$ is extremely fast.
\medskip

Let us illustrate the computation with an example. For clarity, we first compute the covariants and then derive $Q$ and $E$ even though, the user does not have access to the covariants.
In practice, since the coefficients of $Q$ and $E$ are known polynomials in the basis of $65$ invariants for non-hyperelliptic genus $4$ curves of rank $4$, they can be directly evaluated from the invariants of a given example.

\begin{ex}
  Let $\mathcal{C}$ be the projective non-hyperelliptic genus $4$ curve canonically embedded in $\mathbb{P}^3$, defined by the vanishing locus of
    \[Q = XT-YZ,\] and \[E = X^2Y + X^2Z + X^2T + XY^2 + XYZ + XZ^2 + XZT + XT^2 + Y^2Z + YZ^2 +
        YZT + YT^2 + T^3.\]

The quadratic form $Q$ is of rank $4$, thus we pullback through the Segre morphism the cubic form $E$ to a bicubic form $f$ in $x,y$ and $u,v$.
Then, we compute its covariants $c_{31},c_{51,1}, c_{51,2}$ and $c_{51,3}$.

\begin{align*}
    c_{31} &= -44xu - 17xv - 25yu - 17yv,\\
    c_{51,1} &= 9xu - 107xv - 88yu - 24yv,\\
    c_{51,2} &= -620xu - 1937xv - 1129yu + 181yv,\\
    c_{51,3} &= 25889xu - 5563xv - 19056yu + 1328yv.
\end{align*}

We can now compute the equations of $Q$ and $E$ given by Equation~\ref{eq:reconstruction_g4}. We obtain
\begin{align*}
    Q =\  &646X^2 - 6536XY - 130084XZ - 1923144XT - 19264Y^2 - 549500YZ- 6275840YT\\
    & - 4598186Z^2 - 78659100ZT - 143255872T^2\,,\\
        &\\
    E = &-87337008X^3 + 69815520X^2Y - 3596033232X^2Z + 178527014496X^2T - 629045568XY^2\\
    &- 13790445696XYZ - 435571233408XYT - 147774846096XZ^2 + 586163101824XZT\\
    &- 162711651196224XT^2 + 489595536Y^3 + 31071365856Y^2Z+ 625393402416Y^2T\\
    &+ 676666128096YZ^2 + 20257026499008YZT + 246651902537904YT^2 + 4187892749328Z^3\\
    &+ 229585773241440Z^2T + 1868504372517600ZT^2 + 47848070690492688T^3.
\end{align*}

The minimization of the coefficients of non-hyperelliptic curves of genus 4 with integer coefficients is a joint work in progress with Andreas Pieper,
it is a variation on the algorithm of~\cite{elsenhans}.

For this curve, our minimization algorithm returns in half a second the model
\begin{align*}
Q =&\ X^2 - XZ - 2YZ + 2Z^2 - XT - YT - T^2\\
E =&-XY^2 - X^2Z + 3XYZ - 2Y^2Z + 2XZ^2 + Z^3 + X^2T\\ 
&+ 4XZT - 3YZT-2Z^2T + 3XT^2 - 6ZT^2 - 2T^3,
\end{align*}
which has much smaller coefficients.

As expected, the computation of the invariants of the reconstructed curve are equal, up to weighted projective equivalence, to the original ones.
\end{ex}

\begin{rem} 
    There are some instances where the reconstruction algorithm fails, because the automorphism group of the curve is too big. 
    Let \[Q = X^2+Y^2+Z^2+T^2+(X+Y+Z+T)^2,\] and \[E = X^3+Y^3+Z^3+T^3-(X+Y+Z+T)^3.\]
    Then the non-hyperelliptic curve of genus $4$ (of rank $4$) defined by $Q$ and $E$, has automorphism group $S_5$, the biggest possible
    for a curve defined over $\mathbb{C}$. The reconstruction algorithm fails, since most of its invariants vanish.
    The author was not able to find a non-hyperelliptic curve of genus 4 (of rank 4) with automorphisms
    which could be reconstructed using the $4$ covariants above.
\end{rem}

\newpage
\appendix
\section{Covariant tables}

    \renewcommand{\arraystretch}{1.3}
    \begin{table}[h!]
        \centering
        \begin{tabular}{|c|c|}
            \hline
            Covariants & Contravariants \\
            \hline
            & $\sigma = w_2^4[(\bold{F}',\bold{F}')_4](w_0/w_2, w_1/w_2)$ \\
            & $\psi = w_2^6[((\bold{F}', \bold{F}')_2, \bold{F}')_4](w_0/w_2, w_1/w_2)$ \\
            $\bold{H} = (\bold{F}, \bold{F}, \bold{F})_2$ & $\rho = \D(\bold{F}, \psi)$ \\
            $\bold{C_{4,4}} = x_2^4[(\sigma', \sigma')_4](x_0/x_2, x_1/x_2)$ & $c_{5,4} = \D(\bold{F}, \sigma^2)$ \\
            $\bold{C_{5,2}} = \D(\sigma, \bold{H})$ & $c_{10,5} = (\sigma, \psi, c_{5,4})_3$ \\
            $\bold{C_{8,5}} = (\bold{F}, \bold{H}, \bold{C_{4,4}})_3$ & $c_{12,3} = \D(\bold{C_{8,5}}, \sigma^2)$ \\
            $\bold{C_{12,3}} = \D(\rho, \bold{C_{8,5}})$ & $p_0 = \D(\bold{C_{12,3}}, \rho)$ \\
            & $p_1 = \D(\bold{C_{12,3}}, c_{5,4})$ \\
            & $p_2 = \D(\bold{C_{5,2}}, c_{12,3})$ \\
            \hline
        \end{tabular}
        \caption{Covariants (bold) and contravariants used to compute $p_0,p_1$ and $p_2$}
        \label{tab:def_covcont_genus3}
    \end{table}

\renewcommand{\arraystretch}{1.2}
\begin{table}[h!]
    \centering
    \begin{tabular}{|c|c|c|c|c|}
        \hline
        \diagbox[width=5.6em]{degree}{order} & 1 & 2 & 3 & 4\\
        \hline
        1 &  &  & $f$ & \\
        \hline
        2 &  & $h=(f,f)_2$ &  & $j=(f,f)_1$ \\
        \hline
        \multirow{2}{*}{3} & \multirow{2}{*}{$c_{31} = (h,f)_2$} &  & $c_{33,1} = (j,f)_2$ &  \\
        & & & $c_{33,2} = (h,f)_1$ & \\
        \hline
        \multirow{3}{*}{4} &  & $c_{42,1} = (h,h)_1$ &  & \multirow{2}{*}{$c_{44,1} = (c_{33,2},f)_1$} \\
        & & $c_{42,2} = (c_{31}, f)_1$ & & \multirow{2}{*}{$c_{44,2} = ((j,f)_1,f)_2$}\\
        & & $c_{42,3} = (c_{33,2}, f)_2$ & & \\
        \hline
        \multirow{3}{*}{5} & $c_{51, 1} = (c_{42,2}, f)_2$ &  & &  \\
        & $c_{51,2} = (c_{44,1}, f)_3$& &  & \\
        & $c_{51,3} = (c_{44,2}, f)_3$ & &  & \\
        \hline
    \end{tabular}
    \caption{Several covariants of $\Sym^3(W^*)\otimes \Sym^3(W^*)$}
    \label{fig:tab_cov}
\end{table}

\bibliographystyle{elsarticle-harv}
\bibliography{paper}

\begin{thebibliography}{35}
\expandafter\ifx\csname natexlab\endcsname\relax\def\natexlab#1{#1}\fi
\providecommand{\url}[1]{\texttt{#1}}
\providecommand{\href}[2]{#2}
\providecommand{\path}[1]{#1}
\providecommand{\DOIprefix}{doi:}
\providecommand{\ArXivprefix}{arXiv:}
\providecommand{\URLprefix}{URL: }
\providecommand{\Pubmedprefix}{pmid:}
\providecommand{\doi}[1]{\href{http://dx.doi.org/#1}{\path{#1}}}
\providecommand{\Pubmed}[1]{\href{pmid:#1}{\path{#1}}}
\providecommand{\bibinfo}[2]{#2}
\ifx\xfnm\relax \def\xfnm[#1]{\unskip,\space#1}\fi
\bibitem[{Andersen and Jantzen(1984)}]{jantzen}
\bibinfo{author}{Andersen, H.H.}, \bibinfo{author}{Jantzen, J.C.}, \bibinfo{year}{1984}.
\newblock \bibinfo{title}{Cohomology of induced representations for algebraic groups}.
\newblock \bibinfo{journal}{Math. Ann.} \bibinfo{volume}{269}, \bibinfo{pages}{487--525}.
\newblock \DOIprefix\doi{10.1007/BF01450762}.
\bibitem[{Bosma et~al.(1997)Bosma, Cannon and Playoust}]{magma}
\bibinfo{author}{Bosma, W.}, \bibinfo{author}{Cannon, J.}, \bibinfo{author}{Playoust, C.}, \bibinfo{year}{1997}.
\newblock \bibinfo{title}{The {M}agma algebra system. {I}. {T}he user language}.
\newblock \bibinfo{journal}{J. Symbolic Comput.} \bibinfo{volume}{24}, \bibinfo{pages}{235--265}.
\newblock \URLprefix \url{http://dx.doi.org/10.1006/jsco.1996.0125}, \DOIprefix\doi{10.1006/jsco.1996.0125}. \bibinfo{note}{computational algebra and number theory (London, 1993)}.
\bibitem[{Bouchet(2024a)}]{bouchet}
\bibinfo{author}{Bouchet, T.}, \bibinfo{year}{2024}a.
\newblock \bibinfo{title}{Invariants of genus 4 curves}.
\newblock \bibinfo{journal}{Journal of Algebra} \bibinfo{volume}{660}, \bibinfo{pages}{619--644}.
\newblock \URLprefix \url{https://www.sciencedirect.com/science/article/pii/S0021869324004009}, \DOIprefix\doi{https://doi.org/10.1016/j.jalgebra.2024.07.016}.
\bibitem[{Bouchet(2024b)}]{git-reconstruction}
\bibinfo{author}{Bouchet, T.}, \bibinfo{year}{2024}b.
\newblock \bibinfo{title}{Reconstruction}.
\newblock \bibinfo{howpublished}{\url{https://github.com/Thittho/Reconstruction}}.
\bibitem[{Bouyer and Streng(2015)}]{bouyer}
\bibinfo{author}{Bouyer, F.}, \bibinfo{author}{Streng, M.}, \bibinfo{year}{2015}.
\newblock \bibinfo{title}{Examples of cm curves of genus two defined over the reflex field}.
\newblock \bibinfo{journal}{LMS Journal of Computation and Mathematics} \bibinfo{volume}{18}, \bibinfo{pages}{507–538}.
\newblock \DOIprefix\doi{10.1112/S1461157015000121}.
\bibitem[{Brouwer and Popoviciu(2010)}]{popoviciu}
\bibinfo{author}{Brouwer, A.E.}, \bibinfo{author}{Popoviciu, M.}, \bibinfo{year}{2010}.
\newblock \bibinfo{title}{The invariants of the binary decimic}.
\newblock \bibinfo{journal}{J. Symb. Comput.} \bibinfo{volume}{45}, \bibinfo{pages}{837--843}.
\newblock \DOIprefix\doi{10.1016/j.jsc.2010.03.002}.
\bibitem[{Cardona and Quer(2005)}]{cardona}
\bibinfo{author}{Cardona, G.}, \bibinfo{author}{Quer, J.}, \bibinfo{year}{2005}.
\newblock \bibinfo{title}{Field of moduli and field of definition for curves of genus 2}, in: \bibinfo{booktitle}{Computational aspects of algebraic curves. Papers from the conference, University of Idaho, Moscow, ID, USA, May 26--28, 2005}. \bibinfo{publisher}{Hackensack, NJ: World Scientific}, pp. \bibinfo{pages}{71--83}.
\bibitem[{Clebsch(1870)}]{clebsch}
\bibinfo{author}{Clebsch, A.}, \bibinfo{year}{1870}.
\newblock \bibinfo{title}{Zur {Theorie} der bin{\"a}ren algebraischen {Formen}.}
\newblock \bibinfo{journal}{G{\"o}tt. Nachr.} \bibinfo{volume}{1870}, \bibinfo{pages}{405--409}.
\bibitem[{Derksen and Kemper(2015)}]{derksen}
\bibinfo{author}{Derksen, H.}, \bibinfo{author}{Kemper, G.}, \bibinfo{year}{2015}.
\newblock \bibinfo{title}{Computational invariant theory. {With} two appendices by {Vladimir} {L}. {Popov} and an addendum by {Nobert} {A}. {Campo} and {Vladimir} {L}. {Popov}}. volume \bibinfo{volume}{130} of \textit{\bibinfo{series}{Encycl. Math. Sci.}}
\newblock \bibinfo{edition}{2nd enlarged edition} ed., \bibinfo{publisher}{Berlin: Springer}.
\newblock \DOIprefix\doi{10.1007/978-3-662-48422-7}.
\bibitem[{Derksen and Makam(2021)}]{derksen_weyl}
\bibinfo{author}{Derksen, H.}, \bibinfo{author}{Makam, V.}, \bibinfo{year}{2021}.
\newblock \bibinfo{title}{Weyl's polarization theorem in positive characteristic}.
\newblock \bibinfo{journal}{Transform. Groups} \bibinfo{volume}{26}, \bibinfo{pages}{1241--1260}.
\newblock \DOIprefix\doi{10.1007/s00031-020-09559-3}.
\bibitem[{Dixmier(1987)}]{dixmier}
\bibinfo{author}{Dixmier, J.}, \bibinfo{year}{1987}.
\newblock \bibinfo{title}{On the projective invariants of quartic plane curves}.
\newblock \bibinfo{journal}{Adv. Math.} \bibinfo{volume}{64}, \bibinfo{pages}{279--304}.
\newblock \DOIprefix\doi{10.1016/0001-8708(87)90010-7}.
\bibitem[{Dolgachev(2003)}]{dolgachev}
\bibinfo{author}{Dolgachev, I.}, \bibinfo{year}{2003}.
\newblock \bibinfo{title}{Lectures on invariant theory}. volume \bibinfo{volume}{296} of \textit{\bibinfo{series}{Lond. Math. Soc. Lect. Note Ser.}}
\newblock \bibinfo{publisher}{Cambridge: Cambridge University Press}.
\bibitem[{Dolgachev(2012)}]{dolgachev_classical}
\bibinfo{author}{Dolgachev, I.V.}, \bibinfo{year}{2012}.
\newblock \bibinfo{title}{Classical algebraic geometry. {A} modern view}.
\newblock \bibinfo{publisher}{Cambridge: Cambridge University Press}.
\newblock \DOIprefix\doi{10.1017/CBO9781139084437}.
\bibitem[{Domokos(2008)}]{domokos}
\bibinfo{author}{Domokos, M.}, \bibinfo{year}{2008}.
\newblock \bibinfo{title}{Covariants and the no-name {Lemma}}.
\newblock \bibinfo{journal}{J. Lie Theory} \bibinfo{volume}{18}, \bibinfo{pages}{839--849}.
\newblock \URLprefix \url{www.heldermann.de/JLT/JLT18/JLT184/jlt18051.htm}.
\bibitem[{Ehrenborg and Rota(1993)}]{ehrenborg}
\bibinfo{author}{Ehrenborg, R.}, \bibinfo{author}{Rota, G.C.}, \bibinfo{year}{1993}.
\newblock \bibinfo{title}{Apolarity and canonical forms for homogeneous polynomials}.
\newblock \bibinfo{journal}{Eur. J. Comb.} \bibinfo{volume}{14}, \bibinfo{pages}{157--181}.
\newblock \DOIprefix\doi{10.1006/eujc.1993.1022}.
\bibitem[{Elsenhans and Stoll(2023)}]{elsenhans}
\bibinfo{author}{Elsenhans, A.S.}, \bibinfo{author}{Stoll, M.}, \bibinfo{year}{2023}.
\newblock \bibinfo{title}{Minimization of hypersurfaces}.
\newblock \href{http://arxiv.org/abs/2110.04625}{{\tt arXiv:2110.04625}}.
\bibitem[{Girard and Kohel(2006)}]{kohel}
\bibinfo{author}{Girard, M.}, \bibinfo{author}{Kohel, D.R.}, \bibinfo{year}{2006}.
\newblock \bibinfo{title}{Classification of genus 3 curves in special strata of the moduli space}, in: \bibinfo{booktitle}{Algorithmic number theory. 7th international symposium, ANTS-VII, Berlin, Germany, July 23--28, 2006. Proceedings.}. \bibinfo{publisher}{Berlin: Springer}, pp. \bibinfo{pages}{346--360}.
\newblock \DOIprefix\doi{10.1007/11792086}.
\bibitem[{Gordan(1868)}]{gordan}
\bibinfo{author}{Gordan}, \bibinfo{year}{1868}.
\newblock \bibinfo{title}{Proof that each covariant and each invariant of a binary form is an entire function, with numerical coefficients, of finitely many such forms.}
\newblock \bibinfo{journal}{J. Reine Angew. Math.} \bibinfo{volume}{69}, \bibinfo{pages}{323--354}.
\newblock \DOIprefix\doi{10.1515/crll.1868.69.323}.
\bibitem[{Haboush(1975)}]{haboush}
\bibinfo{author}{Haboush, W.J.}, \bibinfo{year}{1975}.
\newblock \bibinfo{title}{Reductive groups are geometrically reductive}.
\newblock \bibinfo{journal}{Ann. Math. (2)} \bibinfo{volume}{102}, \bibinfo{pages}{67--83}.
\newblock \DOIprefix\doi{10.2307/1970974}.
\bibitem[{Harris(1992)}]{harris}
\bibinfo{author}{Harris, J.}, \bibinfo{year}{1992}.
\newblock \bibinfo{title}{Algebraic geometry. {A} first course}. volume \bibinfo{volume}{133} of \textit{\bibinfo{series}{Grad. Texts Math.}}
\newblock \bibinfo{publisher}{Berlin etc.: Springer-Verlag}.
\bibitem[{Kılıçer et~al.(2018)Kılıçer, Labrande, Lercier, Ritzenthaler, Sijsling and Streng}]{kilicer}
\bibinfo{author}{Kılıçer, P.}, \bibinfo{author}{Labrande, H.}, \bibinfo{author}{Lercier, R.}, \bibinfo{author}{Ritzenthaler, C.}, \bibinfo{author}{Sijsling, J.}, \bibinfo{author}{Streng, M.}, \bibinfo{year}{2018}.
\newblock \bibinfo{title}{Plane quartics over {$\mathbb{Q}$} with complex multiplication}.
\newblock \bibinfo{journal}{Acta Arith.} \bibinfo{volume}{185}, \bibinfo{pages}{127--156}.
\newblock \URLprefix \url{https://doi.org/10.4064/aa170227-16-3}, \DOIprefix\doi{10.4064/aa170227-16-3}.
\bibitem[{Lercier and Ritzenthaler(2012)}]{LR}
\bibinfo{author}{Lercier, R.}, \bibinfo{author}{Ritzenthaler, C.}, \bibinfo{year}{2012}.
\newblock \bibinfo{title}{Hyperelliptic curves and their invariants: geometric, arithmetic and algorithmic aspects}.
\newblock \bibinfo{journal}{J. Algebra} \bibinfo{volume}{372}, \bibinfo{pages}{595--636}.
\newblock \DOIprefix\doi{10.1016/j.jalgebra.2012.07.054}.
\bibitem[{Lercier et~al.(2014)Lercier, Ritzenthaler, Rovetta and Sijsling}]{LRR}
\bibinfo{author}{Lercier, R.}, \bibinfo{author}{Ritzenthaler, C.}, \bibinfo{author}{Rovetta, F.}, \bibinfo{author}{Sijsling, J.}, \bibinfo{year}{2014}.
\newblock \bibinfo{title}{Parametrizing the moduli space of curves and applications to smooth plane quartics over finite fields}.
\newblock \bibinfo{journal}{LMS J. Comput. Math.} \bibinfo{volume}{17A}, \bibinfo{pages}{128--147}.
\newblock \DOIprefix\doi{10.1112/S146115701400031X}.
\bibitem[{Lercier et~al.(2020)Lercier, Ritzenthaler and Sijsling}]{LRS}
\bibinfo{author}{Lercier, R.}, \bibinfo{author}{Ritzenthaler, C.}, \bibinfo{author}{Sijsling, J.}, \bibinfo{year}{2020}.
\newblock \bibinfo{title}{Reconstructing plane quartics from their invariants}.
\newblock \bibinfo{journal}{Discrete Comput. Geom.} \bibinfo{volume}{63}, \bibinfo{pages}{73--113}.
\newblock \DOIprefix\doi{10.1007/s00454-018-0047-4}.
\bibitem[{Mestre(1991)}]{mestre}
\bibinfo{author}{Mestre, J.F.}, \bibinfo{year}{1991}.
\newblock \bibinfo{title}{Construction de courbes de genre 2 {\`a} partir de leurs modules}. \bibinfo{publisher}{Birkh{\"a}user Boston}, \bibinfo{address}{Boston, MA}.
\newblock pp. \bibinfo{pages}{313--334}.
\newblock \URLprefix \url{https://doi.org/10.1007/978-1-4612-0441-1_21}, \DOIprefix\doi{10.1007/978-1-4612-0441-1_21}.
\bibitem[{Mumford et~al.(1994)Mumford, Fogarty and Kirwan}]{mumford}
\bibinfo{author}{Mumford, D.}, \bibinfo{author}{Fogarty, J.}, \bibinfo{author}{Kirwan, F.}, \bibinfo{year}{1994}.
\newblock \bibinfo{title}{Geometric invariant theory.}. volume~\bibinfo{volume}{34} of \textit{\bibinfo{series}{Ergeb. Math. Grenzgeb.}}
\newblock \bibinfo{edition}{3rd enl. ed.} ed., \bibinfo{publisher}{Berlin: Springer-Verlag}.
\newblock \URLprefix \url{hdl.handle.net/2433/102881}, \DOIprefix\doi{10.1007/978-3-319-65907-7}.
\bibitem[{Nagata(1964)}]{nagata}
\bibinfo{author}{Nagata, M.}, \bibinfo{year}{1964}.
\newblock \bibinfo{title}{Invariants of a group in an affine ring}.
\newblock \bibinfo{journal}{J. Math. Kyoto Univ.} \bibinfo{volume}{3}, \bibinfo{pages}{369--377}.
\newblock \DOIprefix\doi{10.1215/kjm/1250524787}.
\bibitem[{Noordsij(2022)}]{noordsij}
\bibinfo{author}{Noordsij, J.}, \bibinfo{year}{2022}.
\newblock \bibinfo{title}{Reconstruction of binary quintics.}
\newblock Master's thesis. Leiden University.
\bibitem[{Ohno(2007)}]{ohno}
\bibinfo{author}{Ohno, T.}, \bibinfo{year}{2007}.
\newblock \bibinfo{title}{The graded ring of invariants of ternary quartics i}.
\bibitem[{Olive(2017)}]{olive-gordan}
\bibinfo{author}{Olive, M.}, \bibinfo{year}{2017}.
\newblock \bibinfo{title}{About {Gordan}'s algorithm for binary forms}.
\newblock \bibinfo{journal}{Found. Comput. Math.} \bibinfo{volume}{17}, \bibinfo{pages}{1407--1466}.
\newblock \DOIprefix\doi{10.1007/s10208-016-9324-x}.
\bibitem[{Olive et~al.(2017)Olive, Kolev, Desmorat and Desmorat}]{olive_rec}
\bibinfo{author}{Olive, M.}, \bibinfo{author}{Kolev, B.}, \bibinfo{author}{Desmorat, R.}, \bibinfo{author}{Desmorat, B.}, \bibinfo{year}{2017}.
\newblock \bibinfo{title}{{Harmonic factorization and reconstruction of the elasticity tensor}}, in: \bibinfo{booktitle}{{CFM 2017 - 23{\`e}me Congr{\`e}s Fran{\c c}ais de M{\'e}canique}}.
\newblock \URLprefix \url{https://hal.science/hal-03465304}.
\bibitem[{Olver(1999)}]{olver}
\bibinfo{author}{Olver, P.J.}, \bibinfo{year}{1999}.
\newblock \bibinfo{title}{Classical invariant theory}. volume~\bibinfo{volume}{44} of \textit{\bibinfo{series}{Lond. Math. Soc. Stud. Texts}}.
\newblock \bibinfo{publisher}{Cambridge: Cambridge University Press}.
\bibitem[{Schmidt and Hasse(1937)}]{hasse}
\bibinfo{author}{Schmidt, F.}, \bibinfo{author}{Hasse, H.}, \bibinfo{year}{1937}.
\newblock \bibinfo{title}{Noch eine begründung der theorie der höheren differentialquotienten in einem algebraischen funktionenkörper einer unbestimmten. (nach einer brieflichen mitteilung von f.k. schmidt in jena).}
\newblock \bibinfo{journal}{Journal für die reine und angewandte Mathematik} \bibinfo{volume}{1937}, \bibinfo{pages}{215--237}.
\newblock \URLprefix \url{https://doi.org/10.1515/crll.1937.177.215}, \DOIprefix\doi{doi:10.1515/crll.1937.177.215}.
\bibitem[{Shioda(1993)}]{shioda}
\bibinfo{author}{Shioda, T.}, \bibinfo{year}{1993}.
\newblock \bibinfo{title}{Plane quartics and {Mordell}-{Weil} lattices of type {{\(E_ 7\)}}}.
\newblock \bibinfo{journal}{Comment. Math. Univ. St. Pauli} \bibinfo{volume}{42}, \bibinfo{pages}{61--79}.
\bibitem[{Sylvester and Franklin(1879)}]{sylvester}
\bibinfo{author}{Sylvester, J.J.}, \bibinfo{author}{Franklin, F.}, \bibinfo{year}{1879}.
\newblock \bibinfo{title}{Tables of the generating functions and groundforms for the binary quantic of the first ten orders.}
\newblock \bibinfo{journal}{Am. J. Math.} \bibinfo{volume}{2}, \bibinfo{pages}{223--251}.
\newblock \DOIprefix\doi{10.2307/2369240}.

\end{thebibliography}

\end{document}